\documentclass{amsart}
\usepackage{amsfonts,graphicx,amssymb,stmaryrd,amscd,amsmath,latexsym,amsbsy,tikz,bm}
\numberwithin{equation}{section}
\theoremstyle{plain}

\newtheorem{thm}{Theorem}[section]

\newtheorem{prop}[thm]{Proposition}
\newtheorem{defi}[thm]{Definition}
\newtheorem{lem}[thm]{Lemma}
\newtheorem{cor}[thm]{Corollary}
\newtheorem{eg}[thm]{Example}

\theoremstyle{remark}
\newtheorem{rema}[thm]{Remark}

\newcommand{\Z}{\mathbb{Z}}

\newcommand{\C}{\mathbb{C}}

\newcommand{\Id}{\textup{Id}}
\newcommand{\End}{\textup{End}}
\newcommand{\Tr}{\mathop{\textup{Tr}}}

\DeclareRobustCommand{\SkipTocEntry}[4]{}

\title[Koornwinder polynomials and spin chains]{Koornwinder polynomials and the XXZ spin chain}
\author{Jasper Stokman \& Bart Vlaar}
\address{J.S.: KdV Institute for Mathematics, University of Amsterdam, Science Park 904, 1098 XH Amsterdam, The Netherlands \& IMAPP,  Radboud University, Heyendaalseweg 135, 6525 AJ Nijmegen, The Netherlands.}
\email{j.v.stokman@uva.nl}
\address{B.V.: KdV Institute for Mathematics, University of Amsterdam, Science Park 904, 1098 XH Amsterdam, The Netherlands.}
\email{b.h.m.vlaar@uva.nl}
\subjclass[2000]{}
\begin{document}
\keywords{}
\begin{abstract}
Nonsymmetric Koornwinder polynomials are multivariable extensions of nonsymmetric Askey-Wilson polynomials. 
They naturally arise in the representation theory of (double) affine Hecke algebras. 
In this paper we discuss how nonsymmetric Koornwinder polynomials naturally arise in the theory of the Heisenberg XXZ spin-$\frac{1}{2}$ chain with general reflecting boundary conditions. A central role in this story is played by an explicit two-parameter family of spin representations of the two-boundary Temperley-Lieb algebra. 
These spin representations have three different appearances. Their original definition relates them directly to the XXZ spin chain, in the form of matchmaker representations they relate to Temperley-Lieb loop models in statistical physics, while their realization as principal series representations leads to the link with nonsymmetric Koornwinder polynomials. The nonsymmetric difference Cherednik-Matsuo correspondence allows to construct for special parameter values Laurent-polynomial solutions of the associated reflection quantum KZ equations in terms of nonsymmetric Koornwinder polynomials. We discuss these aspects in detail by revisiting and extending work of De Gier, Kasatani, Nichols, Cherednik, the first author and many others.
\end{abstract}
\maketitle
\setcounter{tocdepth}{1}
\begin{center}
Dedicated to the 80th birthday of Dick Askey.
\end{center}
\tableofcontents
%%%%%%%%%%%%%%%%%%%%%%%%%%%%%%%%%%%%%%%%%%%%%
\section{Introduction}
%%%%%%%%%%%%%%%%%%%%%%%%%%%%%%%%%%%%%%%%%%%%%%
\subsection{Nonsymmetric Koornwinder polynomials}
%%%%%%%%%%%%%%%%%%%%%%%%%%%%%%%%%%%%%%%%%%%%%%%
The four-parameter family of Askey-Wilson polynomials, introduced in 1985 in the famous monograph \cite{AW}, is the most general class of classical $q$-orthogonal polynomials in one variable. Important generalizations have appeared since: the associated Askey-Wilson polynomials \cite{IR}, Rahman's \cite{Ra} biorthogonal rational ${}_{10}\phi_9$
functions, nonsymmetric Askey-Wilson polynomials \cite{NS,Sa}, Askey-Wilson functions \cite{KS}, elliptic analogs of Askey-Wilson polynomials \cite{Sp,Rai}, etc. 

There exist two completely different {\it multivariable} extensions of the Askey-Wilson polynomials. The first is due to Gasper and Rahman \cite{GRBook}. In this case
the multivariable Askey-Wilson polynomials admit explicit multivariate basic hypergeometric series expansions, allowing the corresponding theory to be developed by classical methods.
See \cite{GI} for the corresponding multivariable generalization of the Askey-Wilson function. 

The second type of multivariable generalization of the Askey-Wilson polynomials is part of the Macdonald-Cherednik \cite{Macd, CBook, StBook} theory on root system analogs of continuous $q$-ultraspherical, continuous $q$-Jacobi and Askey-Wilson polynomials.
In this case the multivariable extension of the Askey-Wilson polynomial, the nonsymmetric Askey-Wilson polynomial and the Askey-Wilson function are the Koornwinder polynomial \cite{Ko}, the nonsymmetric Koornwinder polynomial \cite{Sa} and the basic hypergeometric function \cite{St0,St2}, respectively. 
These multivariable extensions depend on five parameters and arise in harmonic analysis on (quantum) symmetric spaces \cite{NDS,Let}, representation theory of affine Hecke algebras \cite{CBook,Sa} and in quantum relativistic integrable one-dimensional
many-body systems of Calogero-Moser type \cite{Ru,vD}. Multivariable extensions of Rahman's \cite{Ra} biorthogonal rational functions and of the  elliptic analogs
of the Askey-Wilson functions are due to Rains \cite{Rai}. 

\subsection{Heisenberg spin chains}
\emph{Spin models} originate in the statistical mechanical study \cite{Ba,Be,He} of magnetism. 
The \emph{Heisenberg spin chain} is a one-dimensional quantum model of magnetism with a quantum spin particle residing at each site of a one-dimensional
lattice. The interaction of the quantum spin particles is given by nearest neighbour spin-spin interaction. 
Assuming that the number of sites is finite, say $n$, boundary conditions need to be imposed to determine how the quantum spin particles at the boundary sites $1$ and $n$ are 
treated in the model. The most investigated choice is to impose \emph{periodic (or closed) boundary conditions}, in which case we assume that the sites $1$ and $n$ are also 
adjacent (the one-dimensional lattice is put on a circle).
In this paper we focus on so-called \emph{reflecting (or open) boundary conditions}, in which case the quantum spin particles at the boundary sites $1$ and $n$ are interacting with reflecting boundaries on each side.

The quantum Hamiltonian $H_{\text{bdy}}$ for the Heisenberg spin-$\frac{1}{2}$ chain with reflecting boundaries has the following form \cite{IK,dVGR,Ne}.
The Pauli \cite{Pauli} spin matrices are
\[ \sigma^X = \begin{pmatrix} 0 & 1 \\ 1 & 0 \end{pmatrix}, \qquad
 \sigma^Y = \begin{pmatrix} 0 & -\sqrt{-1} \\ \sqrt{-1} & 0 \end{pmatrix}, \qquad
\sigma^Z = \begin{pmatrix} 1 & 0 \\ 0 & -1 \end{pmatrix}.  \]
Then $H_{\text{bdy}}$ is the linear operator on 
$(\C^2)^{\otimes n} = \underset{(1)}{\C^2} \otimes \cdots \otimes \underset{(n)}{\C^2}$ given by
\begin{align*} 
H_{\text{bdy}} &= \sum_{i=1}^{n-1} \left( J_X \sigma^X_i  \sigma^X_{i+1} + J_Y \sigma^Y_i  \sigma^Y_{i+1} + J_Z \sigma^Z_i  \sigma^Z_{i+1} \right) + \\
&+ \left(M^l_X \sigma^X_1 + M^l_Y \sigma^Y_1 + M^l_Z \sigma^Z_1\right) 
+ \left(M^r_X \sigma^X_n + M^r_Y \sigma^Y_n + M^r_Z \sigma^Z_n\right),
\end{align*}
where the subindices indicate the tensor leg of $(\C^2)^{\otimes n}$ on which the Pauli matrix is acting. 
Besides the three coupling constants $J_X, J_Y$ and $J_Z$ there are three parameters $M^l_X, M^l_Y$ and $M^l_Z$ at the left reflecting boundary and three parameters $M^r_X, M^r_Y$ and $M^r_Z$ at the right reflecting boundary of the one-dimensional lattice. 
In general, the three coupling constants are distinct, in which case $H_{\text{bdy}}$ corresponds to the XYZ spin-$\frac{1}{2}$ model with general reflecting boundary conditions; in its full generality it was first obtained in \cite{IK}. 
In this paper we consider the special case that $J_X=J_Y\not=J_Z$, in which case the spin chain is the XXZ spin-$\frac{1}{2}$ model with general reflecting boundary conditions. We denote the resulting quantum Hamiltonian by $H_{\text{bdy}}^{XXZ}$. 
Up to rescaling, $H_{\text{bdy}}^{XXZ}$ thus has seven free parameters; also see \cite{dVGR,Ne}.

The XXZ spin-$\frac{1}{2}$ chain with generic reflecting boundary conditions is quantum integrable, in the sense that the quantum Hamiltonian is part of a large commuting set of
linear operators on $(\C^2)^{\otimes n}$ encoded by a transfer operator (see \cite{Ne}). 
For $H_{\text{bdy}}^{XXZ}$ the transfer operator is produced from the standard XXZ spin-$\frac{1}{2}$ solution of the quantum Yang-Baxter equation and three-parameter solutions of the associated left and right reflection equations, see \cite{Ne} and Subsection \ref{Hsection}.
It opens the way to study the spectrum and eigenfunctions of $H_{\text{bdy}}^{XXZ}$ using Sklyanin's \cite{Sk1988} generalization of the algebraic Bethe ansatz to quantum integrable models with reflecting boundary conditions. 

%%%%%%%%%%%%%%%%%%%%%%%%%%%%%%%%%%%%%%%%%
\subsection{Representation theory}
%%%%%%%%%%%%%%%%%%%%%%%%%%%%%%%%%%%%%%%%%%
A representation theoretic context for the five parameter family of Koornwinder polynomials is provided by Letzter's \cite{Let} notion of quantum symmetric pairs. The quantum symmetric pairs pertinent to Koornwinder polynomials are given by a family of coideal subalgebras of the quantized universal enveloping algebra of $\mathfrak{gl}_N$
naturally associated to quantum complex Grassmannians (see \cite{NDS,OS}). For the XXZ spin-$\frac{1}{2}$ spin chain with general reflecting boundary conditions,
a representation theoretic context is provided by coideal subalgebras of the quantum affine algebra of $\widehat{\mathfrak{sl}}_2$ known as $q$-Onsager algebras (see \cite{BK}). $q$-Onsager algebras are examples of quantum {\it affine} symmetric pairs \cite{Kolb}. 

The five-parameter double affine Hecke algebras of type $\widetilde{C}_n$ \cite{Sa} provides another representation theoretic context for Koornwinder polynomials. On the other hand, various special cases of the XXZ spin-$\frac{1}{2}$ chain with general reflecting boundary conditions have been related to the three-parameter affine Hecke algebra of type $\widetilde{C}_n$, which is a subalgebra of the double affine Hecke algebra (see \cite{dGN} and references therein).

It is one of the purposes of this paper to show that the double affine Hecke algebra governs the whole seven-parameter family of XXZ spin-$\frac{1}{2}$ chains with reflecting
boundary conditions. The double affine Hecke algebra gives rise to a Baxterization procedure for affine Hecke algebra representations, producing for a given representation a two-parameter family of solutions of quantum Yang-Baxter and reflection equations (see Subsection \ref{Baxt}). In the case of spin representations, which form a two-parameter family of representations of the affine Hecke algebra of type $\widetilde{C}_n$ on the state space $(\C^2)^{\otimes n}$ \cite{dGP1,dGN}, we end up with solutions depending on seven parameters: three affine Hecke algebra parameters denoted by $\underline{\kappa}=(\kappa_0,\kappa,\kappa_n)$, two Baxterization parameters denoted by $\upsilon_0,\upsilon_n$, and two spin representation parameters denoted by $\psi_0,\psi_n$. 
The resulting transfer operator reproduces $H_{\text{bdy}}^{XXZ}$ as the associated quantum Hamiltonian under an appropriate parameter correspondence. 
The parameters $\underline{\kappa},\upsilon_0,\upsilon_n$ are the five parameters of the Koornwinder polynomials.

The spin representations factor through the two-boundary Temperley-Lieb algebra \cite{dGN}. 
The two-boundary Temperley-Lieb algebra has a natural two-parameter family of representations on the formal vector spaces spanned by two-boundary non-crossing perfect matchings (see Definition \ref{perfectDef}), in which the generators of the algebra act by "matchmakers", see, e.g.,  \cite{dG,dGN} and Subsection \ref{Matchsection}. 
It is known \cite{dGN} that generically the resulting matchmaker representations are isomorphic to the spin representations under a suitable parameter correspondence. 
We provide a new proof for this in Subsection \ref{Matchsection} by constructing an explicit intertwiner. 
The Baxterization of the matchmaker representation leads to a transfer operator acting on the formal vector space of two-boundary non-crossing perfect matchings. 
The associated quantum Hamiltonian corresponds to the Temperley-Lieb loop model (also known as dense loop model) with general open boundary conditions (see Subsection \ref{Hsection}); special cases are discussed in, e.g., \cite{MNGB,dG,dGN}. 
It leads to the conclusion that generically the XXZ spin-$\frac{1}{2}$ chain and the Temperley-Lieb loop model for general reflecting boundary conditions are equivalent under a suitable parameter correspondence.

%%%%%%%%%%%%%%%%%%%%%%%%%%%%%%%%%%%%%%%%%%%%%%%%%%%%
\subsection{Weyl group invariant solutions of the reflection quantum KZ equations} 
%%%%%%%%%%%%%%%%%%%%%%%%%%%%%%%%%%%%%%%%%%%%%%%%%%%%
 
The quantum Knizhnik-Zamolodchikov (KZ) equations are a consistent system of first order linear $q$-difference equations, which were first studied by Smirnov in relation to integrable quantum field theories \cite{Sm}. 
They were related to representation theory of quantum affine algebras by Frenkel and Reshetikhin \cite{FR}. 
Their solutions entail correlation functions for XXZ spin chains with quasi-periodic boundary conditions, see, e.g., \cite{JM}.
More recently, links to algebraic geometry have been exposed \cite{dFZJ2}. 
Generalizations of these equations for arbitrary root systems were constructed by Cherednik \cite{CQKZ,CQKZ3}. 
For type A, the aforementioned Smirnov-Frenkel-Reshetikhin quantum KZ equations are recovered; \emph{reflection} quantum KZ equations are the equations in Cherednik's framework associated to the other classical types \cite[\S 5]{CQKZ}. 
In the context of the spin-$\frac{1}{2}$ Heisenberg chain, the reflection quantum KZ equations were first derived in \cite{JKKKM}. 
 
A choice of solutions of the quantum Yang-Baxter equation and the associated reflection equations gives rise to reflection quantum KZ equations \cite{CQKZ}. 
In case the solutions are obtained from the Baxterization of the spin representations we obtain reflection quantum KZ equations depending on seven parameters. 
These are the reflection quantum KZ equations under investigation in the last section of the paper (Section \ref{Solsection}).
Their solutions are expected to give rise to correlation functions for the spin chain governed by $H_{\text{bdy}}^{XXZ}$. 
The close link to the spin chain is apparent from the fact that the transport operators of the reflection quantum KZ equations at $q=1$ are higher quantum Hamiltonians for the spin chain (see Subsection \ref{spsection}).

For special choices of reflecting boundary conditions methods have been developed to construct explicit solutions of the reflection quantum KZ equations, e.g. in \cite{JKKKM,JKKMW,We,dFZJ,dFZJ2,ZJ2,Pa,RSV}. 
In Section \ref{Solsection} we construct solutions using nonsymmetric Koornwinder polynomials, following the ideas from \cite{Ka,St2}.

The solution space of the reflection quantum KZ equations associated to $H_{\text{bdy}}^{XXZ}$ has a natural action of the Weyl group $W_0$ of type $C_n$, which is the hyperoctahedral group $S_n\ltimes\{\pm 1\}^n$. 
In Section \ref{Solsection} we show that generically, the space of Weyl group invariant solutions is in bijective correspondence to a suitable space of eigenfunctions of the Cherednik-Noumi \cite{N} $Y$-operators. 
The $Y$-operators are commuting scalar-valued $q$-difference-reflection operators generalizing Dunkl operators. The link is provided by the nonsymmetric Cherednik-Matsuo correspondence \cite{St1}. 
This correspondence can be applied in the present context, because the spin representations are examples of principal series representations (see Subsection \ref{PrincipalSection}).
 
The nonsymmetric basic hypergeometric function of Koornwinder type \cite{St0,St2} provides distinguished examples of eigenfunctions of the $Y$-operators. We show that
it produces a nontrivial $W_0$-invariant solution of the reflection quantum KZ equations for generic values of the parameters. It is exactly known under which specialization
of the spectral parameters the nonsymmetric basic hypergeometric function reduces to a nonsymmetric Koornwinder polynomial. It follows that if the parameters satisfy
one of the following two conditions,
\begin{equation*}
\begin{split}
\psi_0\psi_nq^m&=\kappa_0\kappa_n\kappa^{n-1}\qquad\,\,\,\,\,\,\, \textup{for some } m\in\mathbb{Z}_{\geq 1},\\
\psi_0\psi_nq^m&=\kappa_0^{-1}\kappa_n^{-1}\kappa^{1-n}\qquad \textup{for some } m\in\mathbb{Z}_{\leq 0},
\end{split}
\end{equation*}
then nontrivial $W_0$-invariant Laurent polynomial solutions of the reflection quantum KZ equations
exist (see Theorem \ref{mainthm}).
Note that these conditions do not depend on the two Baxterization parameters $\upsilon_0,\upsilon_n$. 
$W_0$-invariant Laurent polynomial solutions are expected to play an important role in generalizations of Razumov-Stroganov conjectures \cite{dGR}, cf., e.g., \cite{dGP2, ZJ1, 
KaPrep}.  

%%%%%%%%%%%%%%%%%%%%%%%%%%%%%%%%%%%%%%%%%%%%%%%%%%%%%%%%%%
\subsection{Outlook}
%%%%%%%%%%%%%%%%%%%%%%%%%%%%%%%%%%%%%%%%%%%%%%%%%%%%%%%%%%%
It is an open problem how the solutions of the reflection quantum KZ equations in terms of nonsymmetric basic hypergeometric functions and nonsymmetric Koornwinder polynomials are related to other constructions of explicit solutions \cite{JKKKM,JKKMW,We,dFZJ,dFZJ2,ZJ2,Pa,RSV}.
Wall-crossing formulas for the reflection quantum KZ equations associated to $H_{\text{bdy}}^{XXZ}$ are in reach by \cite{St1,St3}. 
They are expected to lead to elliptic solutions of quantum dynamical Yang-Baxter equations and associated dynamical reflection equations governing the integrability of elliptic solid-on-solid models with reflecting ends, cf. \cite{Fi} for such models with diagonal reflecting ends. 
Pursuing the techniques of the present paper at the elliptic level, either for the XYZ spin chain or elliptic solid-on-solid models, is expected to lead to connections to elliptic hypergeometric functions \cite{Ra,Sp,SpSurvey,SpInt}. 
Only some first steps have been made here, see, e.g., \cite{ZJell, Tak}.

%%%%%%%%%%%%%%%%%%%%%%%%%%%%%%%%%%%%%%%%%%%%%%%%%%%%%%%%%%%
\subsection{Acknowledgements} We are grateful to P. Baseilhac, J. de Gier, B. Nienhuis, N. Resheti\-khin and P. Zinn-Justin for stimulating discussions. 
The work of B.V. was supported by an NWO Free Competition grant ("Double affine Hecke algebras, Integrable Models and Enumerative Combinatorics").
%%%%%%%%%%%%%%%%%%%%%%%%%%%%%%%%%%%%%%%%%%%%%%%%%%%%%%%%%%%

%%%%%%%%%%%%%%%%%%%%%%%%%%%%%%%%%%%%%%%%%%%%%%%%%%%%%%%%%%%
\subsection*{Notational conventions}
%%%%%%%%%%%%%%%%%%%%%%%%%%%%%%%%%%%%%%%%%%%%%%%%%%%%%%%%%%%%
Throughout the paper $n\in\mathbb{Z}_{\geq 2}$. 
Linear operators of $\mathbb{C}^2$ are represented as $2\times 2$-matrices with respect to a fixed ordered basis $(v_+,v_-)$ of $\mathbb{C}^2$, and linear operators of $\mathbb{C}^2\otimes\mathbb{C}^2$ are represented as $4\times 4$-matrices with respect to the ordered
basis $(v_+\otimes v_+,v_+\otimes v_-,v_-\otimes v_+,v_-\otimes v_-)$ of $\mathbb{C}^2\otimes\mathbb{C}^2$.

%%%%%%%%%%%%%%%%%%%%%%%%%%%%%%%%%%%%%%%%%%%%%%%%%%%%%%%%%%%
\section{Groups and algebras of type \except{toc}{$\widetilde{C}_n$}\for{toc}{$\tilde{C}_n$}}  \label{algebrasection}
%%%%%%%%%%%%%%%%%%%%%%%%%%%%%%%%%%%%%%%%%%%%%%%%%%%%%%%%%%%

%%%%%%%%%%%%%%%%%%%%%%%%%%%%%%%%%%%%%%%%%%%%%%%%%%%%%%%%%
\subsection{Affine braid group}
%%%%%%%%%%%%%%%%%%%%%%%%%%%%%%%%%%%%%%%%%%%%%%%%%%%%%%%%%
\begin{defi}
The affine braid group of type $\widetilde{C}_n$ is the group $\mathcal{B}$ with generators $\sigma_0,\sigma_1,\ldots,\sigma_n$ subject to the braid relations
\begin{align*}
\sigma_0\sigma_1\sigma_0\sigma_1&=\sigma_1\sigma_0\sigma_1\sigma_0,\\
\sigma_{n-1}\sigma_n\sigma_{n-1}\sigma_n&=\sigma_n\sigma_{n-1}\sigma_n\sigma_{n-1},\\
\sigma_i\sigma_{i+1}\sigma_i&=\sigma_{i+1}\sigma_i\sigma_{i+1}, && 1\leq i<n-1, \\
\sigma_i \sigma_j &= \sigma_j \sigma_i, && 1 \leq i,j \leq n, \, |i-j|>1. 
\end{align*}
\end{defi}
%%%%%%%%%%%%%%%%%%%%%%%%%%%%%%%%%%%%%%%%%%%%%%%%%%%%%%%%%
It can be topologically realized as follows (see the type $\widetilde{C}_n$ case in \cite[\S 4]{Al}).
Consider the line segments $P_l=\{0\}\times\{0\}\times [0,1]$ and $P_r=\{0\}\times\{n+1\}\times [0,1]$ in $\mathbb{R}^3$.
Let $\widetilde{\mathcal{B}}$ be the group of $n$-braids in the strip $\bigl(\mathbb{R}^2\times [0,1]\bigr)\setminus (P_l\cup P_r)$, with the $n$-braids attached to the floor at $(0,i,0)$ ($1\leq i\leq n$) and to the ceiling at $(0,i,1)$ ($1\leq i\leq n$). The group operation $\sigma\tau$ is putting the $n$-braid $\tau$ on top of $\sigma$ and shrinking the height by isotopies. 
The group isomorphism $\mathcal{B}\simeq\widetilde{\mathcal{B}}$ is induced by the identification
\begin{gather*}
\sigma_i = 
\begin{minipage}[c]{36mm}
\begin{tikzpicture}[scale=36/84]
\draw[gray,line width = 5pt] (0,0) -- (0,4) node[at start,below] {$P_l$};
\draw[gray,line width = 5pt] (7,0) -- (7,4) node[at start,below] {$P_r$};
\draw[very thick] (1,0) -- (1,4) node[at start, below] {1};
\draw(2,2) node{\ldots};
\draw[very thick] (3,0) .. controls (3,1.6) .. node[at start, below]{$i \vphantom{1}$} (3.5,2) .. controls (4,2.4) .. (4,4) ;
\draw[white,line width = 5pt] (4,0)  .. controls (4,1.6) .. (3.5,2) .. controls (3,2.4) .. (3,4);
\draw[very thick] (4,0)  .. controls (4,1.6) .. node[at start, below]{$i\!\!+\!\!1$} (3.5,2) .. controls (3,2.4) .. (3,4);
\draw(5,2) node{\ldots};
\draw[very thick] (6,0) -- (6,4) node[at start, below] {$n \vphantom{1}$};
\end{tikzpicture}
\end{minipage}
\qquad  1 \leq i < n, \\
\sigma_0 = 
\begin{minipage}[c]{32mm}
\begin{tikzpicture}[scale=32/72]
\draw[gray,line width = 5pt] (0,0) -- (0,4) node[at start,below] {$P_l$};
\draw[gray,line width = 5pt] (5,0) -- (5,4) node[at start,below] {$P_r$};
\draw[white,line width = 5pt] (1,0) .. controls (1,1.6) .. (0,1.6) .. controls (-1,1.6) .. (-1,2);
\draw[very thick] (1,0) .. controls (1,1.6) .. node[at start, below] {1} (0,1.6) .. controls (-1,1.6) .. (-1,2) .. controls (-1,2.4) .. (-.3,2.4);
\draw[very thick] (.3,2.4) .. controls (1,2.4) .. (1,4);
\draw[very thick] (2,0) -- (2,4) node[at start, below] {2};
\draw(3,2) node{\ldots};
\draw[very thick] (4,0) -- (4,4) node[at start, below] {$n \vphantom{1}$};
\end{tikzpicture}
\end{minipage}
\qquad
\sigma_n = 
\begin{minipage}[c]{32mm} 
\begin{tikzpicture}[scale=32/72]
\draw[gray,line width = 5pt] (0,0) -- (0,4) node[at start,below] {$P_l$};
\draw[gray,line width = 5pt] (5,0) -- (5,4) node[at start,below] {$P_r$};
\draw[very thick] (1,0) -- (1,4) node[at start, below] {1};
\draw(2,2) node{\ldots};
\draw[very thick] (3,0) -- (3,4) node[at start, below] {$\! n\!\!-\!\!1 \,$};
\draw[very thick] (4,0) .. controls (4,1.6) .. node[at start, below] {$n \vphantom{1}$} (4.7,1.6);
\draw[white,line width = 5pt] (6,2) .. controls (6,2.4) .. (5,2.4) .. controls (4,2.4) .. (4,4);
\draw[very thick] (5.3,1.6) .. controls (6,1.6) .. (6,2) .. controls (6,2.4) .. (5,2.4) .. controls (4,2.4) .. (4,4);
\end{tikzpicture}
\end{minipage}
\end{gather*}
where the right hand sides are the braid diagram projections on the $\mathbb{R}^2\simeq\{0\}\times\mathbb{R}^2$-plane. 
For $1 \leq i \leq n$, it is easy to check topologically that the elements
\begin{align*}
g_i &:=\sigma_{i-1}^{-1}\cdots\sigma_1^{-1}\sigma_0\sigma_1\cdots\sigma_{n-1}\sigma_n\sigma_{n-1}\cdots\sigma_i \\
&= \begin{minipage}[c]{60mm}
\begin{tikzpicture}[scale=.5]
\draw[gray,line width = 5pt] (0,0) -- (0,4) node[at start,below] {$P_l$};
\draw[gray,line width = 5pt] (8,0) -- (8,4) node[at start,below] {$P_r$};
\draw[very thick] (.3,2) -- (1.3,2);
\draw[very thick, dashed] (1.3,2) -- (2.7,2);
\draw[very thick] (2.7,2) -- (5.3,2);
\draw[very thick, dashed] (5.3,2) -- (6.7,2);
\draw[very thick] (6.7,2) -- (7.7,2);
\draw[white,line width = 5pt] (3,1) -- (0,1) .. controls (-1,1) .. (-1,1.5) ;
\draw[very thick] (4,0) .. controls (4,1) .. node[at start, below]{$i \vphantom{1}$} (3,1) -- (2.7,1);
\draw[very thick,dashed] (2.7,1) -- (1.3,1);
\draw[very thick] (1.3,1) -- (0,1) .. controls (-1,1) .. (-1,1.5) .. controls (-1,2) .. (-.3,2) ; 
\draw[white,line width = 5pt] (1,0) -- (1,4);
\draw[white,line width = 5pt] (3,0) -- (3,4);
\draw[white,line width = 5pt] (5,0) -- (5,4);
\draw[white,line width = 5pt] (7,0) -- (7,4);
\draw[very thick] (1,0) -- (1,4) node[at start, below] {1};
\draw[very thick] (3,0) -- (3,4) node[at start, below] {$i\!\!-\!\!1$};
\draw[very thick] (5,0) -- (5,4) node[at start, below] {$i\!\!+\!\!1$};
\draw[very thick] (7,0) -- (7,4) node[at start, below] {$n  \vphantom{1}$};
\draw[white,line width = 5pt] (9,2.5) .. controls (9,3) .. (8,3) -- (5,3) .. controls (4,3) .. (4,4);
\draw[very thick] (8.3,2) .. controls (9,2) .. (9,2.5) .. controls (9,3) .. (8,3) -- (6.7,3);
\draw[very thick,dashed] (6.7,3) -- (5.3,3);
\draw[very thick] (5.3,3) -- (5,3) .. controls (4,3) .. (4,4);
\end{tikzpicture}
\end{minipage}
\end{align*}
of the affine braid group $\mathcal{B}$ pairwise commute.

%%%%%%%%%%%%%%%%%%%%%%%%%%%%%%%%%%%%%%%%%%%%%
\subsection{The affine Weyl group}
%%%%%%%%%%%%%%%%%%%%%%%%%%%%%%%%%%%%%%%%%%%%%%
The affine Weyl group $W$ of type $\widetilde{C}_n$ is the quotient of $\mathcal{B}$ by the relations $\sigma_i^2=1$ ($0\leq i\leq n$).
The generators $\sigma_0,\ldots,\sigma_n$ descend to group generators of $W$, which we denote by $s_0,\ldots,s_n$. The affine Weyl group
$W$ is a Coxeter group with Coxeter generators $s_0,\ldots,s_n$.
Write $\tau_i$ for the commuting elements in $W$ corresponding to $g_i\in\mathcal{B}$ under the canonical surjection $\mathcal{B}\twoheadrightarrow W$, so that
\begin{equation}\label{taui}
\tau_i=s_{i-1}\cdots s_1s_0s_1\cdots s_{n-1}s_ns_{n-1}\cdots s_{i}.
\end{equation}

The affine Weyl group $W$ acts faithfully on $\mathbb{R}^n$ by affine linear transformations via
\begin{equation*}
\begin{split}
s_0(x_1,\ldots,x_n)&=(1-x_1,x_2,\ldots,x_n),\\
s_i(x_1,\ldots,x_n)&=(x_1,\ldots,x_{i-1},x_{i+1},x_i,x_{i+2},\ldots,x_n),\\
s_n(x_1,\ldots,x_n)&=(x_1,\ldots,x_{n-1},-x_n)
\end{split}
\end{equation*}
for $1\leq i<n$. Note that the $s_j$ act on $\mathbb{R}^n$ as orthogonal reflections in affine hyperplanes. We sometimes call $s_0,\ldots,s_n$ the simple reflections of $W$.
The subgroup $W_0$ of $W$ generated by $s_1,\ldots,s_n$, acting by permutations and sign changes of the
coordinates, is isomorphic to $S_n\ltimes (\pm 1)^n$ with $S_n$ the symmetric group in $n$ letters (which, in turn, is isomorphic
to the subgroup generated by $s_1,\ldots,s_{n-1}$). Note furthermore that
\[
\tau_i(x_1,\ldots,x_n)=(x_1,\ldots,x_{i-1},x_i+1,x_{i+2},\ldots,x_n),\qquad 1\leq i\leq n,
\]
hence the abelian subgroup of $W$ generated by $\tau_1,\ldots,\tau_n$ is isomorphic to $\mathbb{Z}^n$. 
We write $\tau( \lambda):=\tau_1^{\lambda_1}\cdots\tau_n^{\lambda_n}\in W$ for $ \lambda=(\lambda_1,\ldots,\lambda_n)\in\mathbb{Z}^n$.
We have $W \cong W_0\ltimes \mathbb{Z}^n$, with $W_0$ acting on $\mathbb{Z}^n$ by permutations and sign changes of the coordinates.

%%%%%%%%%%%%%%%%%%%%%%%%%%%%%%%%%%%%%%%%%%%%%%%%%%
\subsection{Affine Hecke algebra}
%%%%%%%%%%%%%%%%%%%%%%%%%%%%%%%%%%%%%%%%%%%%%%%%%%%
Fix $\underline{\kappa}=(\kappa_0,\kappa_1,\ldots,\kappa_n)\in\bigl(\mathbb{C}^*\bigr)^{n+1}$ with 
$\kappa_1=\kappa_2=\cdots=\kappa_{n-1}$. We write $\kappa$ for the value $\kappa_i$ ($1\leq i<n$). 

The affine Hecke algebra $H(\underline{\kappa})$ of type $\widetilde{C}_n$ is the quotient of the group algebra $\mathbb{C}[\mathcal{B}]$ of the affine braid group $\mathcal{B}$ by the two-sided ideal generated by the elements $(\sigma_j-\kappa_j)(\sigma_j+\kappa_j^{-1})$ for $0\leq j \leq n$. 
The generators $\sigma_0,\ldots,\sigma_n$ descend to algebraic generators of $H(\underline{\kappa})$, which we denote by $T_0,\ldots,T_n$.

The quadratic relation $(T_j-\kappa_j)(T_j+\kappa_j^{-1})=0$ in $H(\underline{\kappa})$ implies that
$T_j$ is invertible in $H(\underline{\kappa})$ with inverse $T_j-\kappa_j+\kappa_j^{-1}$.
 
For $\underline{\kappa}=(1,1,\ldots,1)$ the affine Hecke algebra $H(\underline{\kappa})$ is the group 
algebra $\mathbb{C}[W]$ of the affine Weyl group $W$ of type $C_n$.
 
Let $Y_i$ be the element in $H(\underline{\kappa})$ corresponding to $g_i\in\mathcal{B}$ under the canonical surjection
$\mathbb{C}[\mathcal{B}]\twoheadrightarrow H(\underline{\kappa})$. Then
\[
Y_i=T_{i-1}^{-1}\cdots T_1^{-1}T_0T_1\cdots T_{n-1}T_nT_{n-1}\cdots T_i,\qquad 1\leq i\leq n.
\]
The $Y_i$ pairwise commute in $H(\underline{\kappa})$. They are sometimes called Murphy elements (cf., e.g., \cite[Def. 2.8]{dGN}).
In the context of representation theory of affine Hecke algebras, they naturally arise in the Bernstein-Zelevinsky presentation
of the affine Hecke algebra, see \cite{Lu}.

We write
\[
Y^{ \lambda}:=Y_1^{\lambda_1}Y_2^{\lambda_2}\cdots Y_n^{\lambda_n},\qquad  \lambda=(\lambda_1,\ldots,\lambda_n)
\in\mathbb{Z}^n.
\]
The elements $T_1,\ldots,T_n$ and $Y^{ \lambda}$ ($ \lambda \in \mathbb{Z}^n$) generate $H(\underline{\kappa})$.

%%%%%%%%%%%%%%%%%%%%%%%%%%%%%%%%%%%%%%%%%%%%%%%%%%%%%%%%
\subsection{The two-boundary Temperley-Lieb algebra} \label{2BTL}
%%%%%%%%%%%%%%%%%%%%%%%%%%%%%%%%%%%%%%%%%%%%%%%%%%%%%%%%
In analogy to the setup for the affine Hecke algebra, fix $\underline{\delta} = (\delta_0,\delta_1,\cdots, \delta_n) \in\bigl(\mathbb{C}^*\bigr)^{n+1}$ with $\delta_1 = \delta_2 = \cdots = \delta_{n-1} =: \delta$. 
The \emph{two-boundary Temperley-Lieb algebra} TL$(\underline{\delta})$ \cite{dGN}, also known as the \emph{open Temperley-Lieb algebra} \cite{dGP1}, is the unital associative algebra over $\C$ with generators $e_0,\ldots,e_n$ satisfying
\begin{align}
&& e_i^2 &= \delta_i e_i \label{TLquadraticrelation} \\
&& e_i e_{i \pm 1} e_i &= e_i && 1 \leq i < n, \label{TLnoncommutingrelation} \\
&& e_i e_j &= e_j e_i,&& |i-j|>1. \label{TLcommutingrelation}
\end{align}
See \cite{dGN,dGP1} for a diagrammatic realization of $\textup{TL}(\underline{\delta})$. 

To generic affine Hecke algebra parameters $\underline{\kappa}$ we associate two-boundary Temperley-Lieb parameters $\underline{\delta}$ by
\begin{equation} \label{eq:deltas} 
\delta_j = -\frac{\kappa_j+\kappa_j^{-1}}{\kappa \kappa_j^{-1} +\kappa^{-1} \kappa_j}, 
\qquad\quad \delta =  -(\kappa+\kappa^{-1}) 
\end{equation}
for $j=0$ and $j=n$. In the remainder of the paper we always assume that the affine Hecke algebra parameters and the two-boundary Temperley-Lieb algebra 
parameters are matched in this way. Then there exists a unique surjective algebra homomorphism $\phi: H(\underline \kappa) \twoheadrightarrow TL(\underline \delta)$ such that 
\begin{equation} \label{phiimages} 
\phi(T_j) = 
\begin{cases}  \kappa_j + (\kappa \kappa_j^{-1} +\kappa^{-1} \kappa_j) e_j,\qquad & j=0,n, \\ 
\kappa+e_j,\qquad & 1 \leq j < n,
\end{cases} 
\end{equation}
see \cite[Prop. 2.13]{dGN}. In particular, $\textup{TL}(\underline{\delta})$ is isomorphic to a quotient of $H(\underline{\kappa})$.
The kernel of $\phi$ can be explicitly described, see \cite[Lem. 2.15]{dGN}.

%%%%%%%%%%%%%%%%%%%%%%%%%%%%%%%%%%%%%%%%%%%%%%%%%%
\section{Representations}  \label{repsection}
%%%%%%%%%%%%%%%%%%%%%%%%%%%%%%%%%%%%%%%%%%%%%%%%%%%
\subsection{Spin representation}\label{Spinrep}
In the following lemma we give a two-parameter family of representations of the two-boundary Temperley-Lieb
algebra $\textup{TL}(\underline{\delta})$
on the state space $\bigl(\mathbb{C}^2\bigr)^{\otimes n}$ of the Heisenberg XXZ spin-$\frac{1}{2}$ chain. It includes
the one-parameter family of representations that has been intensively studied in the physics literature (see, e.g., \cite[\S 2.3]{dGN}
and references therein).

We use the standard tensor leg notations for linear operators on $\bigl(\mathbb{C}^2\bigr)^{\otimes n}$. Note also the convention on the matrix notation
for linear operators on $\mathbb{C}^2\otimes\mathbb{C}^2$ as given at the end of the introduction.
%%%%%%%%%%%%%%%%%%%%%%%%%%%%%%%%%%%%%%%%%%%%%%%%%%%%
\begin{lem}
Let the free parameters $\underline{\delta}$ of the two-boundary Temperley-Lieb algebra be given in terms of the
generic affine Hecke algebra parameters $\underline{\kappa}$ by \eqref{eq:deltas}. Let $\psi_0,\psi_n\in\mathbb{C}^*$.
There exists a unique algebra homomorphism 
\[
\hat{\rho}=\hat{\rho}^{\underline{\kappa}}_{\psi_0,\psi_n}:
\textup{TL}(\underline{\delta})\rightarrow\textup{End}_{\mathbb{C}}\bigl((\mathbb{C}^2)^{\otimes n}\bigr)
\]
such that 
\begin{align*}
\hat \rho(e_0) &= \frac{1}{\kappa \kappa_0^{-1} + \kappa^{-1} \kappa_0} \begin{pmatrix} -\kappa_0^{-1} & \psi_0 \\ \psi_0^{-1} & -\kappa_0 \end{pmatrix}_1 \\
\hat \rho(e_i) &= \begin{pmatrix} 0 & 0 & 0 & 0 \\ 0 & -\kappa & 1 & 0 \\ 0 & 1 & -\kappa^{-1} & 0 \\ 0 & 0 & 0 & 0 \end{pmatrix}_{i , i\!+\!1} \\
\hat \rho(e_n) &= \frac{1}{\kappa \kappa_n^{-1} + \kappa^{-1} \kappa_n} 
\begin{pmatrix} -\kappa_n & \psi_n^{-1} \\ \psi_n & -\kappa_n^{-1} \end{pmatrix}_n. 
\end{align*}
\end{lem}
\begin{proof}
This is a straightforward verification.
\end{proof}
%%%%%%%%%%%%%%%%%%%%%%%%%%%%%%%%%%%%%%%%%%%%%%%%%%%%%
We lift $\hat{\rho}$ to a representation 
\[
\rho=\rho^{\underline{\kappa}}_{\psi_0,\psi_n}: H(\underline{\kappa})\rightarrow
\textup{End}_{\mathbb{C}}\bigl((\mathbb{C}^2)^{\otimes n}\bigr)
\]
of the affine Hecke algebra via the surjection $\phi: H(\underline{\kappa})\twoheadrightarrow \textup{TL}(\underline{\delta})$, so
$\rho:=\hat{\rho}\circ\phi$. Then
\[
\rho(T_0):=\bar K_1,\qquad
\rho(T_i):=(\Upsilon\circ P)_{i,i+1},\qquad
\rho(T_n):=K_n
\]
($1\leq i<n$) with 
\begin{equation*}
\begin{split}
\bar K&:=\left(\begin{matrix} \kappa_0-\kappa_0^{-1} & \psi_0\\
\psi_0^{-1} & 0\end{matrix}\right),\qquad
K:=\left(\begin{matrix} 0 &\psi_n^{-1}\\
\psi_n & \kappa_n-\kappa_n^{-1}\end{matrix}\right),\\
\Upsilon&:=\left(\begin{matrix} \kappa & 0 & 0 & 0\\
0 & 1 & 0 & 0\\
0 & \kappa-\kappa^{-1} & 1 & 0\\
0 & 0 & 0 & \kappa\end{matrix}\right)
\end{split}
\end{equation*}
and $P:\mathbb{C}^2\otimes\mathbb{C}^2\rightarrow \mathbb{C}^2\otimes\mathbb{C}^2$ the flip operator
$P(v\otimes w)=P(w\otimes v)$. The braid relations of $T_0,\ldots,T_n$ imply that $\Upsilon$ is a solution of the quantum Yang-Baxter equation
and that $K$ and $\bar K$ are solutions to associated reflection equations (all equations without spectral parameter).
The solution $\Upsilon$ is the well-known solution of the quantum Yang-Baxter equation arising from the universal $R$-matrix of the
quantized universal enveloping algebra $\mathcal{U}_{1/\kappa}(\mathfrak{sl}_2)$ acting on $\mathbb{C}^2\otimes\mathbb{C}^2$, with
$\mathbb{C}^2$ viewed as the vector representation of $\mathcal{U}_{1/\kappa}(\mathfrak{sl}_2)$.

The representations $\rho$ and $\hat{\rho}$ are called spin representations of $H(\underline{\kappa})$ and
$\textup{TL}(\underline{\delta})$, respectively.
 
 %%%%%%%%%%%%%%%%%%%%%%%%%%%%%%%%%%%%%%%%%%%%%%%%%%%%%%%%%%%%%%%%%%%%
 \subsection{Principal series representation}\label{PrincipalSection}
 %%%%%%%%%%%%%%%%%%%%%%%%%%%%%%%%%%%%%%%%%%%%%%%%%%%%%%%%%%%%%%%%%%%%%
 One of the main aims of this paper is to relate solutions of reflection quantum Knizhnik-Zamolodchikov equations for Heisenberg XXZ
 spin-$\frac{1}{2}$ chains with boundaries to Macdonald-Koornwinder type functions. 
Such a connection is known in the context of so-called principal series representations of affine Hecke algebras; see Subsection \ref{dCMsection} and references therein. 
Principal series representations are a natural class of finite dimensional representations of the affine Hecke algebra, which we define now first. 
In the second part of this subsection we show that the spin representation $\rho$ is isomorphic to a principal series representation.
  
Set $T:=\bigl(\mathbb{C}^*\bigr)^n$.  Note that $W_0\simeq S_n\ltimes (\pm 1)^n$ acts on $T$ by permutations and inversions of the coordinates.
%%%%%%%%%%%%%%%%%%%%%%%%%%%%%%%%%%%%%%%%%%%%%%%%%%%%%%%%%%%%%%%
\begin{defi}\label{basicdefprincipal}
Let $I\subseteq\{1,\ldots,n\}$.
\begin{enumerate}
\item[{\bf (i)}] Let $H_I(\underline{\kappa})$ be the subalgebra generated by $T_i$ ($i\in I$) and $Y^\lambda$ ($\lambda\in\mathbb{Z}^n$).
\item[{\bf (ii)}] Let $T_I^{\underline{\kappa}}$ be the elements $\gamma\in T$ satisfying $\gamma_n=\kappa_0\kappa_n$ if $n\in I$ and satisfying
$\gamma_i/\gamma_{i+1}=\kappa^2$ if $i\in \{1,\ldots,n-1\}\cap I$.
\end{enumerate}
\end{defi}
%%%%%%%%%%%%%%%%%%%%%%%%%%%%%%%%%%%%%%%%%%%%%%%%%%%%%%%%%%%%%%%

%%%%%%%%%%%%%%%%%%%%%%%%%%%%%%
\begin{lem}
Let $\gamma\in T_I^{\underline{\kappa}}$.
There exists a unique algebra map $\chi_{I,\gamma}^{\underline{\kappa}}: H_I(\underline{\kappa})\rightarrow\mathbb{C}$ satisfying
\begin{equation*}
\begin{split}
\chi_{I,\gamma}^{\underline{\kappa}}(T_i)&=\kappa_i,\qquad i\in I,\\
\chi_{I,\gamma}^{\underline{\kappa}}(Y^\lambda)&=\gamma^\lambda,\qquad \lambda\in\mathbb{Z}^n.
\end{split}
\end{equation*}
\end{lem}
%%%%%%%%%%%%%%%%%%%%%%%%%%%%%%%%%%%%%%%%%%%%%%%%%%%%%%%%%%%%%%%%%%%%%
\begin{proof}
The proof is a straightforward adjustment of the proof of \cite[Lem. 2.5(i)]{St1}.
\end{proof}
%%%%%%%%%%%%%%%%%%%%%%%%%%%%%%%%%%%%%%%%%%%%%%%%%%%%%%%%%%%%%%%%%%%%%%
We write $\mathbb{C}_{I,\gamma}$ for $\mathbb{C}$ viewed as $H_I(\underline{\kappa})$-module with representation map $\chi_{I,\gamma}^{\underline{\kappa}}$.
%%%%%%%%%%%%%%%%%%%%%%%%%%%%%%%%%%%%%%%%%%%%%%%%%%%%%%%%%%%%%%%%%%%%%%
\begin{defi}
Let $I\subseteq\{1,\ldots,n\}$ and $\gamma\in T_I^{\underline{\kappa}}$.
The associated principal series module is the induced $H(\underline{\kappa})$-module $M_I^{\underline{\kappa}}(\gamma):=
\textup{Ind}_{H_I(\underline{\kappa})}^{H(\underline{\kappa})}(\mathbb{C}_{I,\gamma})$. We write $\pi_{I,\gamma}^{\underline{\kappa}}$ for the corresponding representation map.
\end{defi}
%%%%%%%%%%%%%%%%%%%%%%%%%%%%%%%%%%%%%%%%%%%%%%%%%%%%%%%%%%%%%%%%%%%%%%%%
Concretely, $M_I^{\underline{\kappa}}(\gamma)=H(\underline{\kappa})\otimes_{H_I(\underline{\kappa})}\mathbb{C}_{I,\gamma}$ with the $H(\underline{\kappa})$-action
given by
\[
\pi_{I,\gamma}^{\underline{\kappa}}(h)(h^\prime\otimes_{H_I(\underline{\kappa})}1):=(hh^\prime)\otimes_{H_I(\underline{\kappa})}1,\qquad h,h^\prime\in H(\underline{\kappa}).
\]
The principal series module $M_I^{\underline{\kappa}}(\gamma)$ is finite dimensional. In fact, 
\[
\textup{Dim}_{\mathbb{C}}\bigl(M_I^{\underline{\kappa}}(\gamma)\bigr)=\#W_0/\#W_{0,I}
\]
with $W_{0,I}$ the subgroup of $W_0$ generated by the $s_i$ ($i\in I$). The $W_0$-orbit $W_0\gamma$ of $\gamma$ in $T$ is called the
central character of $M_I^{\underline{\kappa}}(\gamma)$.

%%%%%%%%%%%%%%%%%%%%%%%%%%%%%%%%%%%%%%%%%%%%%%%%%
\begin{prop} \label{rhopiequivalence}
Let $\psi_0,\psi_n\in\mathbb{C}^*$. 
We have
\[
\rho_{\psi_0,\psi_n}^{\underline{\kappa}} \simeq \pi_{J, \zeta}^{\underline{\kappa}}
\]
with
\begin{equation*}
\begin{split}
J&:=\{1,2,\ldots,n-1\},\\
 \zeta&:=(\psi_0\psi_n\kappa^{n-1},\psi_0\psi_n\kappa^{n-3},\ldots,\psi_0\psi_n\kappa^{1-n}).
\end{split}
\end{equation*}
\end{prop}
%%%%%%%%%%%%%%%%%%%%%%%%%%%%%%%%%%%%%%%%%%%%%%%%%
\begin{proof}
Note that $\zeta_i/\zeta_{i+1}=\kappa^2$ for $i\in J$, hence $\pi_{J, \zeta}^{\underline{\kappa}}$ is well defined.

Observe that $\rho_{\psi_0,\psi_n}^{\underline{\kappa}}$ is a cyclic $H(\underline{\kappa})$-representation with cyclic vector $v_+^{\otimes n}$. Note furthermore that
\[
\rho_{\psi_0,\psi_n}^{\underline{\kappa}}(T_i)v_+^{\otimes n}=\kappa v_+^{\otimes n},\qquad i\in J.
\]
In addition, for $i\in\{1,\ldots,n\}$,
\begin{equation*}
\begin{split}
\rho_{\psi_0,\psi_n}^{\underline{\kappa}}(Y_i)v_+^{\otimes n}&=\kappa^{n-i}\rho_{\psi_0,\psi_n}^{\underline{\kappa}}(T_{i-1}^{-1}\cdots T_1^{-1}T_0\cdots T_{n-1}T_n)v_+^{\otimes n}\\
&=\psi_n\kappa^{n-i}\rho_{\psi_0,\psi_n}^{\underline{\kappa}}(T_{i-1}^{-1}\cdots T_1^{-1}T_0\cdots T_{n-1})(v_+^{\otimes (n-1)}\otimes v_-)\\
&=\psi_n\kappa^{n-i}\rho_{\psi_0,\psi_n}^{\underline{\kappa}}(T_{i-1}^{-1}\cdots T_1^{-1}T_0)(v_-\otimes v_+^{\otimes (n-1)})\\
&=\psi_0\psi_n\kappa^{n-i}\rho_{\psi_0,\psi_n}^{\underline{\kappa}}(T_{i-1}^{-1}\cdots T_1^{-1})v_+^{\otimes n}\\
&=\psi_0\psi_n\kappa^{n-2i+1}v_+^{\otimes n}.
\end{split}
\end{equation*}
Hence $\rho_{\psi_0,\psi_n}^{\underline{\kappa}}(h)v_+^{\otimes n}=\chi_{J, \zeta}^{\underline{\kappa}}(h)v_+^{\otimes n}$ for $h\in H_I(\underline{\kappa})$. We thus have a surjective 
linear map $M_J^{\underline{\kappa}}( \zeta)\twoheadrightarrow \bigl(\mathbb{C}^2\bigr)^{\otimes n}$ mapping 
$h\otimes_{H_I(\underline{\kappa})}1$ to $\rho_{\psi_0,\psi_n}^{\underline{\kappa}}(h)v_+^{\otimes n}$ for all $h\in H(\underline{\kappa})$, which intertwines the
$\pi_{J,\zeta}^{\underline{\kappa}}$-action on $M_J^{\underline{\kappa}}(\gamma)$ with the $\rho_{\psi_0,\psi_n}^{\underline{\kappa}}$-action
on $\bigl(\mathbb{C}^2\bigr)^{\otimes n}$. A dimension count shows that it is an isomorphism.
\end{proof}
%%%%%%%%%%%%%%%%%%%%%%%%%%%%%%%%%%%%%%%%%%%%%%%%%%%%%%

 %%%%%%%%%%%%%%%%%%%%%%%%%%%%%%%%%%%%%%%%%%%%%%%%%%
\subsection{Matchmaker representation}\label{Matchsection}
%%%%%%%%%%%%%%%%%%%%%%%%%%%%%%%%%%%%%%%%%%%%%%%%%%%
In this subsection we revisit some of the key results from \cite{dGN} and reestablish them by different methods.

The diagrammatic realization of (boundary) Temperley-Lieb algebras (see, e.g., \cite{dG,dGP1,dGN}) gives rise to natural examples of Temperley-Lieb algebra actions on linear combinations of non-crossing matchings (or equivalently, link patterns) in which the Temperley-Lieb algebra generators act as matchmakers. 
We give a two-parameter family of such matchmaker representations of $\textup{TL}(\underline{\delta})$. 
We reestablish the result from \cite{dGN} that this family is generically isomorphic to the family of spin representations of $\textup{TL}(\underline{\delta})$ by constructing an explicit intertwiner  involving a subtle combinatorial expression given in terms of the various orientations of the pertinent non-crossing matchings. 
The origin of such explicit intertwiners traces back to the explicit link between loop models and the XXZ spin chain from \cite[\S 8]{MNGB}.
In the quasi-periodic context, related to the affine Hecke algebra of type $\widetilde{A}$ and the affine Temperley-Lieb algebra, such intertwiners were considered in \cite[\S 3.2]{ZJ2} and \cite[Prop. 3.2]{MDSA}.

Let $\mathcal{S}$ be the set of unordered pairs $\{i,j\}$ ($i\not=j$) of $[n+1]:=\{0,\ldots,n+1\}$ and denote by $\mathcal{P}(\mathcal{S})$ the power set  of $\mathcal{S}$. 
If $\mathfrak{p}\in\mathcal{P}(\mathcal{S})$ then $i\in [n+1]$ is said to be matched to $j\in [n+1]$ if $\{i,j\}\in\mathfrak{p}$.
%%%%%%%%%%%%%%%%%%%%%%%%%%%%%%%%%%%%%%%%%%%%%%%%%%%%%%%%%%%%%%%%
\begin{defi}\label{perfectDef}
$\mathfrak{p}\in\mathcal{P}(\mathcal{S})$ is called a two-boundary non-crossing perfect matching if the following conditions hold:
\begin{enumerate}
\item[\bf{1.}] Each $i\in\{1,\ldots n\}$ is matched to exactly one $j\in [n+1]$.
\item[{\bf 2.}] There are no pairs $\{i,j\}, \{k,l\}\in\mathfrak{p}$ with $i<k<j<l$.
\item[{\bf 3.}] $\{0,n+1\}\not\in\mathfrak{p}$.
\end{enumerate}
We write $\mathcal{M}$ for the set of two-boundary non-crossing perfect matchings of $[n+1]$.
\end{defi}
%%%%%%%%%%%%%%%%%%%%%%%%%%%%%%%%%%%%%%%%%%%%%%%%%%%%%%%%%%%%%%%%%
Note that the definition of two-boundary non-crossing perfect matching allows $0$ and $n+1$ to have multiple matchings; these boundary points may also be unmatched.
We can give a diagrammatic representation of $\mathfrak{p}\in\mathcal{P}(\mathcal{S})$ by connecting $(i,0)$ and $(j,0)$ by an arc in the upper half plane for all $\{i,j\}\in\mathfrak{p}$. 
If $\mathfrak{p}\in\mathcal{M}$ then this can be done in such a way that the arcs do not intersect except possibly at the boundary endpoints $(0,0)$ and $(n+1,0)$.
For $\mathfrak p \in \mathcal{M}$ and $1 \leq i \leq n$ we write $m_i(\mathfrak p)\in [n+1]\setminus \{i\}$ for the unique element such that $\{i,m_i(\mathfrak{p})\}\in\mathfrak{p}$. 
Next we set 
\[ \alpha_i(\mathfrak p) = \begin{cases} - &\quad \hbox{ if } m_i(\mathfrak p)<i , \\ + &\quad \hbox{ if } m_i(\mathfrak p)>i. \end{cases} 
\]
The map $\nu: \mathcal{M} \to \{ +,- \}^n$ defined by $\nu(\mathfrak p):=(\alpha_1(\mathfrak p), \ldots, \alpha_n(\mathfrak p))$, is a bijection
(cf., e.g., \cite{PoThesis}). Consequently $\#\mathcal{M}=2^n$. 

Let $\mathbb{C}[\mathcal{M}]$ be the formal vector space over $\mathbb{C}$ with basis the two-boundary non-crossing perfect matchings
$\mathfrak{p}\in\mathcal{M}$.
We define now an action of the two-boundary Temperley-Lieb algebra $\textup{TL}(\underline{\delta})$ on $\mathbb{C}[\mathcal{M}]$, depending on two parameters
$\beta_0,\beta_1\in\mathbb{C}^*$, in which the generators $e_j$ act by so-called matchmakers. The idea is as follows. If $1\leq i<n$ and $\mathfrak{p}\in\mathcal{M}$
then the action of $e_i$ on $\mathfrak{p}$ replaces all pairs containing $i$ and/or $i+1$ and adds pairs $\{i,i+1\}$ and $\{m_{i}(\mathfrak{p}),m_{i+1}(\mathfrak{p})\}$.
When this does not produce a two-boundary  non-crossing perfect matching, either because $m_i(\mathfrak{p})$ and $m_{i+1}(\mathfrak{p})$ are both boundary points or because
$m_i(\mathfrak{p})=i+1$, then only the pair $\{i,i+1\}$ is matched and the omission of the second pair
is accounted for by a suitable multiplicative constant (in this case, the deleted pair is either an arc from $j$ to $j$ for some $0\leq j\leq n+1$ or an arc from $0$ to $n+1$).
Thus the action of $e_i$ on a two-boundary non-crossing perfect matching $\mathfrak{p}$ always has the effect that it matches up $i$ and $i+1$.
A similar matchmaker interpretation is given for the boundary operators $e_0$ and $e_n$, which match up $0$ with $1$
and $n$ with $n+1$, respectively.

We now give the precise formulation of the resulting matchmaker representation of $\textup{TL}(\underline{\delta})$.
For $\mathfrak{p}\in\mathcal{P}(\mathcal{S})$ and $1\leq i<n$ write $\mathfrak{p}_i$ to be the element in $\mathcal{P}(\mathcal{S})$ obtained from
$\mathfrak{p}$ by removing any pairs containing $i$ and/or $i+1$. Similarly, for $j=0,n$ we write $\mathfrak{p}_j$ to be the element in $\mathcal{P}(\mathcal{S})$ obtained from $\mathfrak{p}$ by removing any pair containing $1,n$, respectively. 
For $i \in \Z$ write $\text{pty}(i) = 0$ if $i$ is even and $\text{pty}(i)=1$ if $i$ is odd. 
%%%%%%%%%%%%%%%%%%%%%%%%%%%%%%%%%%%%%%%%%%%%%%%%%%%%%%%%%%%%%%%%
\begin{thm}
Fix $\beta_0,\beta_1\in\mathbb{C}^*$. There exists a unique algebra homomorphism
\[
\omega = \omega^{\underline \delta}_{\beta_0,\beta_1}: \textup{TL}(\underline{\delta})\rightarrow\textup{End}_{\mathbb{C}}\bigl(\mathbb{C}[\mathcal{M}]\bigr)
\]
such that $e_j\mathfrak{p}:=\omega(e_j)\mathfrak{p}$ for $0\leq j\leq n$ and $\mathfrak{p}\in\mathcal{M}$ is given by
\[ e_i \mathfrak p =\begin{cases}
\delta \mathfrak p & \text{if } m_i(\mathfrak p)=i+1, \\
\delta_0^{\text{pty}(i-1)} (\mathfrak p_i \cup  \{ i,i+1\}) & \text{if } m_i(\mathfrak p)=0=m_{i+1}(\mathfrak p), \\
\delta_n^{\text{pty}(n+1-i)} (\mathfrak p_i \cup \{ i,i+1 \}) & \text{if } m_i(\mathfrak p)=n+1=m_{i+1}(\mathfrak p), \\
\beta_{\text{pty}(i)} (\mathfrak p_i \cup \{ i,i+1 \}) & \text{if } m_i(\mathfrak p)=0, \, m_{i+1}(\mathfrak p)=n+1, \\
\mathfrak p_i \cup \{ i,i+1 \}\cup \{ m_i(\mathfrak p),m_{i+1}(\mathfrak p) \}, & \text{otherwise},
\end{cases}
\]
for $1 \leq i < n$, and
\begin{align*}
e_0 \mathfrak p &= \begin{cases} 
\delta_0 \mathfrak p,&\quad \text{if } m_1(\mathfrak p)=0, \\
\beta_0 ( \mathfrak p_0 \cup \{ 0,1\}) &\quad \text{if } m_1(\mathfrak p)=n+1, \\
\mathfrak p_0 \cup \{ 0,1 \}\cup\{ 0, m_\mathfrak p(1) \}, &\quad \text{otherwise},
 \end{cases}\\
e_n \mathfrak p &= \begin{cases} 
\delta_n \mathfrak p &\quad \text{if } m_n(\mathfrak p)=n+1, \\
\beta_{\text{pty}(n)} ( \mathfrak p_n \cup \{ n,n+1 \}) &\quad \text{if } m_n(\mathfrak p)=0, \\
\mathfrak p_n \cup \{ n,n+1 \}\cup\{  m_n(\mathfrak p),n+1 \}, &\quad \text{otherwise}.
 \end{cases}
\end{align*}
\end{thm}
%%%%%%%%%%%%%%%%%%%%%%%%%%%%%%%%%%%%%%%%%%%%%%%%%%%%%%%%%%
\begin{proof}
It is sufficient to verify that the defining relations \eqref{TLquadraticrelation}-\eqref{TLcommutingrelation} of $\textup{TL}(\underline \delta)$ are 
satisfied by the matchmakers $\omega(e_j)$. 
This can be done by a straightforward case-by-case analysis, relying in part on the fact that if $m_i(\mathfrak p) = j$ for $1 \leq i,j \leq n$
and $\mathfrak{p}\in\mathcal{M}$, then 
$\text{pty}(i) \ne \text{pty}(j)$. 
An instructive example is
the case $e_0 e_n \mathfrak p = e_n e_0 \mathfrak p$ for $\mathfrak p \in \mathcal{M}$ such that $\{1,n\} \in \mathfrak p$.
Then
\begin{gather*}
e_0 e_n \mathfrak p = e_0 (\mathfrak p_n \cup \{ 1,n+1 \}\cup\{n,n+1\}) = 
\beta_0 (\mathfrak p_{0,n} \cup \{0,1\}\cup\{n,n+1\}) \\
e_n e_0 \mathfrak p = e_n (\mathfrak p_0 \cup \{ 0,1 \}\cup\{ 0,n\}) =
\beta_{\text{pty}(n)} (\mathfrak p_{0,n} \cup  \{0,1\}\cup\{n,n+1\}),
\end{gather*}
where $\mathfrak p_{i,j}:=(\mathfrak p_i)_j = (\mathfrak p_j)_i$ for $0 \leq i \leq n$. Since $\{1,n\}\in\mathfrak{p}$
the parities of $1$ and $n$ must be different, hence $\text{pty}(n)=0$. 
Hence $e_0 e_n \mathfrak p = e_n e_0 \mathfrak p$, as requested.
\end{proof}
%%%%%%%%%%%%%%%%%%%%%%%%%%%%%%%%%%%%%%%%%%%%%%%%%%%%%%%%%%%
See, e.g., \cite{dG,dGN} for a discussion of the diagrammatic representation of the action of Temperley-Lieb algebras
on matchings. It is very helpful for direct computations. An example of such a diagrammatic computation 
for our representation $\omega^{\underline{\delta}}_{\beta_0,\beta_1}$ of $\textup{TL}(\underline{\delta})$ on $\mathbb{C}[\mathcal{M}]$  ($n=3$) is
\[ \begin{array}{ccccc}
\begin{minipage}[c]{20mm} \begin{tikzpicture}[scale = 1/2]
\draw[gray] (0,-4) -- (0,2); \draw[gray] (4,-4) -- (4,2); \draw[gray] (0,0) -- (4,0); \draw[gray] (0,-2) -- (4,-2); \draw[gray] (0,-4) -- (4,-4);
\fill (1,0) circle (3pt); \fill (2,0) circle (3pt); \fill (3,0) circle (3pt); \fill (1,-2) circle (3pt); \fill (2,-2) circle (3pt); \fill (3,-2) circle (3pt);
\fill (1,-4) circle (3pt); \fill (2,-4) circle (3pt); \fill (3,-4) circle (3pt);
\draw (0,-.75) .. controls (1,-.75) .. (1,0) .. controls (1,1) .. (4,1);
\draw (2,0) .. controls (2,1) and (3,1) .. (3,0) -- (3,-2) .. controls (3,-3) and (2,-3) .. (2,-2) -- (2,0);
\draw (0,-1.25) .. controls (1,-1.25) .. (1,-2) -- (1,-4);
\draw (2,-4) .. controls (2,-3) and (3,-3) .. (3,-4);
\end{tikzpicture} \end{minipage} \vspace{1mm}
&=&
\beta_0 \, \begin{minipage}[c]{20mm} \begin{tikzpicture}[scale = 1/2]
\draw[gray] (0,-2) -- (0,2); \draw[gray] (4,-2) -- (4,2); \draw[gray] (0,-2) -- (4,-2);
\fill (1,-2) circle (3pt); \fill (2,-2) circle (3pt); \fill (3,-2) circle (3pt);
\draw (2,0) .. controls (2,1) and (3,1) .. (3,0) .. controls (3,-1) and (2,-1) .. (2,0);
\draw (0,.75) .. controls (1,.75) .. (1,0) -- (1,-2);
\draw (2,-2) .. controls (2,-1) and (3,-1) .. (3,-2);
\end{tikzpicture} \end{minipage}
&=&
\beta_0 \delta \, \begin{minipage}[c]{20mm} \begin{tikzpicture}[scale = 1/2]
\draw[gray] (0,0) -- (0,2); \draw[gray] (4,0) -- (4,2); \draw[gray] (0,0) -- (4,0);
\fill (1,0) circle (3pt); \fill (2,0) circle (3pt); \fill (3,0) circle (3pt);
\draw (0,.75) .. controls (1,.75) .. (1,0) ;
\draw (2,0) .. controls (2,1) and (3,1) .. (3,0);
\end{tikzpicture} \end{minipage} \\
e_2 e_0 \bigl( (+,+,-) \bigr)&=& \beta_0 e_2 \bigl( (-,+,-) \bigr) &=& \beta_0 \delta (-,+,-) \end{array}
\]
where the left and right vertical line should be thought of as the boundary point $0$ and $n+1$, respectively, and where we have represented $\mathfrak{p}\in\mathcal{M}$ by its image under the bijection $\nu:\mathcal{M}\overset{\sim}{\longrightarrow} \{+,-\}^3$.

Note that the spin representation $\hat\rho^{\underline \kappa}_{\psi_0,\psi_n}$ is isomorphic to $\hat\rho^{\underline\kappa}_{\psi_0^\prime,\psi_n^\prime}$
if $\psi_0\psi_n=\psi_0^\prime\psi_n^\prime$ in view of Proposition \ref{rhopiequivalence}.
We now show that the $2^n$-dimensional $\textup{TL}(\underline{\delta})$-representations
$\hat\rho^{\underline{\kappa}}_{\psi_0,\psi_n}$ and $\omega^{\underline{\kappa}}_{\beta_0,\beta_1}$ are isomorphic for generic parameters if
\begin{equation}\label{parameterproduct}
\beta_0 \beta_1 = \psi_0^{-1} \psi_n^{-1}  \cdot
\begin{cases}  \frac{(1+\kappa_0 \kappa_n^{-1} \psi_0 \psi_n)(1+\kappa_0^{-1} \kappa_n \psi_0 \psi_n)}{(\kappa \kappa_0^{-1} + \kappa^{-1} \kappa_0) (\kappa \kappa_n^{-1} + \kappa^{-1} \kappa_n)} &\quad \hbox{ if } n \text{ odd} \\
\frac{(1-\kappa^{-1} \kappa_0 \kappa_n \psi_0 \psi_n)(1-\kappa \kappa_0^{-1} \kappa_n^{-1} \psi_0 \psi_n)}{(\kappa \kappa_0^{-1} + \kappa^{-1} \kappa_0) (\kappa \kappa_n^{-1} + \kappa^{-1} \kappa_n)} &\quad \hbox{ if } n \text{ even} \end{cases} 
\end{equation}
by writing down an explicit intertwiner (for a different approach, see \cite{dGN}). 

To define the intertwiner we write for $h \in \{0,1\}$ and $\mathfrak p \in \mathcal{M}$,
\begin{gather*}
L_{0,h}(\mathfrak p) := \# \left\{ \{ i,0\} \in \mathfrak p \, \mid \, 1 \leq i \leq n \text{ and } \text{pty}(i)=h \right\}, \\
L_{n,h}(\mathfrak p) := \# \left\{ \{ i,n+1\} \in \mathfrak p \, \mid \, 1 \leq i \leq n \text{ and } \text{pty}(i)=h \right\}. \end{gather*}
A straightforward induction argument on $n$ shows that
\begin{equation}\label{Lsum}\sum_{j \in \{0,n\}} \sum_{h \in \{0,1\}} (-1)^h L_{j,h}(\mathfrak p) = -\text{pty}(n)
\end{equation}
for all $\mathfrak p \in \mathcal{M}$. Define 
\[ M(\mathfrak p) :=  \prod_{j \in \{0,n\} } \prod_{ h \in \{0,1\}} M_{j,h}^{L_{j,h}(\mathfrak p)}, \]
with $M_{j,h}\in\mathbb{C}^*$ 
satisfying
\begin{equation} \label{Mrelations}
\begin{gathered}
M_{j,0} M_{j,1} = \psi_j^{-1} (\kappa \kappa_j^{-1} + \kappa^{-1} \kappa_j)^{-1}, \qquad j\in\{0,n\}, \\
M_{0,0} M_{n,1} =  \beta_0 \cdot \begin{cases} (1+\kappa_0 \kappa_n^{-1} \psi_0 \psi_n)^{-1} &\quad \hbox{ if } n \textup{ odd} \\ 
(1-\kappa^{-1} \kappa_0 \kappa_n \psi_0 \psi_n)^{-1} &\quad \hbox{ if } n \textup{ even}. \end{cases}
\end{gathered}
\end{equation}
By \eqref{Lsum} the conditions \eqref{Mrelations} fix $M(\mathfrak p)$ up to a multiplicative constant.

Next, define an \emph{oriented two-boundary non-crossing perfect matching} as a two-boundary non-crossing perfect matching $\mathfrak{p}$
with a chosen ordering of each pair in $\mathfrak{p}$. 
We write ordered pairs as $(i,j)$ and say that the pair (or, with the diagrammatic realization in mind, the associated
connecting arc) is oriented from $i$ to $j$. We write $\vec{\mathcal{M}}$ for the set of oriented two-boundary non-crossing perfect matchings.
Let $\text{Forg}: \vec{\mathcal{M}}\twoheadrightarrow\mathcal{M}$ be the canonical surjective map which forgets the orientation.
For $\vec{\mathfrak{p}}\in\vec{\mathcal{M}}$ we write $\mathfrak{p}:=\textup{Forg}(\vec{\mathfrak{p}})$.

Given $\vec{\mathfrak{p}} \in \overrightarrow{\mathcal{M}}$ and $i=1,\ldots,n$ we set $r_i(\vec{\mathfrak{p}})$ equal to $+$ or $-$ if the oriented arc connecting $i$ and $m_i(\mathfrak{p})$ is oriented outwards or inwards at $i$, respectively. In other words, $r_i(\vec{\mathfrak{p}})=+$ if 
$(i,m_i(\mathfrak{p}))\in\vec{\mathfrak{p}}$ and $r_i(\mathfrak{p})=-$ if $(m_i(\mathfrak{p}),i) \in\vec{\mathfrak{p}}$.
We define a mapping $v: \overrightarrow{\mathcal{M}} \hookrightarrow (\C^2)^{\otimes n}$ by
\[ v(\vec{\mathfrak p}) = v_{r_1(\vec{\mathfrak{p}})} \otimes \cdots \otimes v_{r_n(\vec{\mathfrak{p}})}. \]
As an example with $n=4$ we have $v(\{ (0,1), (3,2), (5,4)  \}) = v_- \otimes v_- \otimes v_+ \otimes v_-$. 

For $h \in \{0,1\}$ and $\vec{\mathfrak p} \in \overrightarrow{\mathcal{M}}$ we set
\begin{gather*}
N_{0,h}(\vec{\mathfrak p}) := \# \left\{ (i,0) \in \vec{\mathfrak p} \, \mid \, 1 \leq i \leq n \text{ and } \text{pty}(i)=h \right\}, \\
N_{n,h}(\vec{\mathfrak p}) := \# \left\{ (n+1,i) \in \vec{\mathfrak p} \, \mid \, 1 \leq i \leq n \text{ and } \text{pty}(n+1-i)=h  \right\}.
\end{gather*}
Finally, for $\vec{\mathfrak p} \in \overrightarrow{\mathcal{M}}$ define
\[ \text{or}(\vec{\mathfrak p}) = \# \left\{ (j,i) \in \vec{\mathfrak p} \, \mid \, 1 \leq i<j \leq n \right\} + N_{0,0}(\vec{\mathfrak p}) + N_{n,0}(\vec{\mathfrak p}). \]

%%%%%%%%%%%%%%%%%%%%%%%%%%%%%%%%%%%%%%%%%%%%%%%%%%%%%%%%%
\begin{thm} \label{isomorphicrepresentations}
Suppose that the parameters $(\underline{\delta},\beta_0,\beta_1)$ are related to
$(\underline{\kappa},\psi_0,\psi_n)$ by \eqref{eq:deltas} and \eqref{parameterproduct}.

For generic parameters 
the linear map $\Psi:=\Psi^{\underline{\kappa}}_{\psi_0,\psi_n}: \mathbb{C}[\mathcal{M}]\rightarrow\bigl(\mathbb{C}^2\bigr)^{\otimes n}$, defined by
\[
\Psi(\mathfrak{p}):=M(\mathfrak{p})\sum_{\vec{\mathfrak p} \in \textup{Forg}^{-1}(\mathfrak p)} (-\kappa)^{-\text{or}(\vec{\mathfrak p})} \biggl( \prod_{j \in \{0,n\}} (-\kappa_j)^{N_{j,0}(\vec{\mathfrak p})-N_{j,1}(\vec{\mathfrak p})} \psi_j^{N_{j,0}(\vec{\mathfrak p})+N_{j,1}(\vec{\mathfrak p})} \biggr) v(\vec{\mathfrak p})
\]
for $\mathfrak{p}\in\mathcal{M}$, defines an isomorphism of $\textup{TL}(\underline{\delta})$-modules with respect to the $\textup{TL}(\underline{\delta})$-actions
$\omega^{\underline{\delta}}_{\beta_0,\beta_1}$ on $\mathbb{C}[\mathcal{M}]$ and $\hat{\rho}^{\underline{\kappa}}_{\psi_0,\psi_n}$ on $\bigl(\mathbb{C}^2\bigr)^{\otimes n}$.
\end{thm}
%%%%%%%%%%%%%%%%%%%%%%%%%%%%%%%%%%%%%%%%%%%%%%%%%%%%%%%%%%
\begin{proof}
The intertwining property $\Psi(\omega^{\underline{\delta}}_{\beta_0,\beta_1}(e_j)\mathfrak{p})=
\hat{\rho}^{\underline{\kappa}}_{\psi_0,\psi_n}(e_j)(\mathfrak{p})$ for $0\leq j\leq n$ and $\mathfrak{p}\in\mathcal{M}$
follows by a careful case-by-case check.  

To show that $\Psi$ is an isomorphism for generic parameters, it suffices to show 
the modified linear map $\overline{\Psi}: \mathbb{C}[\mathcal{M}]\rightarrow\bigl(\mathbb{C}^2\bigr)^{\otimes n}$, defined by
\[
\overline{\Psi}(\mathfrak{p}):=\sum_{\vec{\mathfrak p} \in \textup{Forg}^{-1}(\mathfrak p)} (-\kappa)^{-\text{or}(\vec{\mathfrak p})} \biggl( \prod_{j \in \{0,n\}} (-\kappa_j)^{N_{j,0}(\vec{\mathfrak p})-N_{j,1}(\vec{\mathfrak p})} \psi_j^{N_{j,0}(\vec{\mathfrak p})+N_{j,1}(\vec{\mathfrak p})} \biggr) v(\vec{\mathfrak p})
\]
for $\mathfrak{p}\in\mathcal{M}$, is generically a linear isomorphism. Since $\overline{\Psi}$ depends polynomially on the parameters $\kappa^{-1},\psi_0,\psi_n$, it suffices to show that it is a linear isomorphism when $\psi_0=\psi_n=\kappa^{-1}=0$.  Then 
\[
\overline{\Psi}_{\psi_0=\psi_n=\kappa^{-1}=0}(\mathfrak{p})=v_{\alpha_1(\mathfrak{p})}\otimes v_{\alpha_2(\mathfrak{p})}\otimes\cdots\otimes v_{\alpha_n(\mathfrak{p})}
\]
for $\mathfrak{p}\in\mathcal{M}$, which defines a linear isomorphism $\overline{\Psi}: \mathbb{C}[\mathcal{M}]\overset{\sim}{\longrightarrow}
\bigl(\mathbb{C}^2\bigr)^{\otimes n}$  since the map $\nu(\mathfrak{p}):=(\alpha_1(\mathfrak{p}),\ldots,\alpha_n(\mathfrak{p}))$ is a bijection $\nu: \mathcal{M}\overset{\sim}{\longrightarrow}\{+,-\}^n$.
\end{proof}
%%%%%%%%%%%%%%%%%%%%%%%%%%%%%%%%%%%%%%%%%%%%%%%%%%%%%%%%%%%

%%%%%%%%%%%%%%%%%%%%%%%%%%%%%%%%%%%%%%%%%%%%%%%%%%%%%%%%%%
\section{Integrable models}  \label{integrablesection}
%%%%%%%%%%%%%%%%%%%%%%%%%%%%%%%%%%%%%%%%%%%%%%%%%%%%%%%%%%
\subsection{Baxterization}\label{Baxt}
%%%%%%%%%%%%%%%%%%%%%%%%%%%%%%%%%%%%%%%%%%%%%%%%%%%%%%%%%%
In \cite[\S 5]{Jo} Jones analysed when braid group representations can be "Baxterized". 
Baxterization means that the action of the $\sigma_i$ ($1\leq i<n$) extend to $R$-matrices with spectral parameter, the essential building blocks for integrable lattice models of vertex type. 
The result \cite[Prop 2.18]{dGN} solves the Baxterization problem for the spin representations of two-boundary Temperley-Lieb algebras. 
However its proof, which is by direct computations, does not give insight in the Baxterization procedure. 

Cherednik's theory on double affine Hecke algebras gives a natural Baxterization procedure for arbitrary representations of
affine Hecke algebras. 
We explain this here briefly for the affine Hecke algebra $H(\underline{\kappa})$ of type $\widetilde{C}_n$.
It reproduces for the spin representations the earlier mentioned result \cite[Prop. 2.18]{dGN} of de Gier and Nichols.
%%%%%%%%%%%%%%%%%%%%%%%%%%%%%%%%%%%%%%%%%%%%%%%%%
\begin{prop} \label{Baxterization}
Let $\pi: H(\underline{\kappa})\rightarrow\End_{\mathbb{C}}(V)$ be a representation of $H(\underline{\kappa})$ and set
\begin{equation*}
\begin{split}
K_0^V(x;\underline{\kappa};\upsilon_0,\upsilon_n)&:=\frac{\pi(T_0^{-1})+(\upsilon_0^{-1}-\upsilon_0)x-x^2\pi(T_0)}{\kappa_0^{-1}(1-\kappa_0\upsilon_0x)(1+\kappa_0\upsilon_0^{-1}x)},\\
R_i^V(x;\underline{\kappa};\upsilon_0,\upsilon_n)&:=\frac{\pi(T_i^{-1})-x\pi(T_i)}{\kappa^{-1}(1-\kappa^2x)},\qquad 1\leq i<n,\\
K_n^V(x;\underline{\kappa};\upsilon_0,\upsilon_n)&:=\frac{\pi(T_n^{-1})+(\upsilon_n^{-1}-\upsilon_n)x-x^2\pi(T_n)}{\kappa_n^{-1}(1-\kappa_n\upsilon_nx)
(1+\kappa_n\upsilon_n^{-1}x)}
\end{split}
\end{equation*}
as elements in $\mathbb{C}(x)\otimes_{\mathbb{C}}\End_{\mathbb{C}}(V)$. 
Then we have, as endomorphisms of $V$ with rational dependence on the spectral parameters,
\begin{equation*}
\begin{split}
K_0^V(x)R_1^V(xy)K_0^V(y)R_1^V(y/x)&=R_1^V(y/x)K_0^V(y)R_1^V(xy)K_0^V(x),\\
R_i^V(x)R_{i+1}^V(xy)R_i^V(y)&=R_{i+1}^V(y)R_i^V(xy)R_{i+1}^V(x),\qquad 1\leq i <n-1,\\
K_n^V(y)R_{n-1}^V(xy)K_n^V(x)R_{n-1}^V(x/y)&=R_{n-1}^V(x/y)K_n^V(x)R_{n-1}^V(xy)K_n^V(y),
\end{split}
\end{equation*}
and 
\begin{equation*}
\begin{split}
K_0^V(x)K_0^V(x^{-1})&=\Id_V=K_n^V(x)K_n^V(x^{-1}),\\
R_i^V(x)R_i^V(x^{-1})&=\Id_V,\qquad 1\leq i<n,
\end{split}
\end{equation*}
as well as 
\begin{equation*}
\begin{split}
\lbrack K_0^V(x),R_i^V(y)\rbrack&=0,\qquad 2\leq i<n,\\
\lbrack K_n^V(x),R_i^V(y)\rbrack&=0,\qquad 1\leq i<n-1,\\
\lbrack R_i^V(x),R_j^V(y)\rbrack&=0,\qquad |i-j|\geq 2,\\
\lbrack K_0^V(x),K_n^V(y)\rbrack&=0,
\end{split}
\end{equation*}
where we have suppressed the dependence on $\underline{\kappa}$, $\upsilon_0$ and $\upsilon_n$.
The original action of the affine Hecke algebra is recovered by specializing the spectral parameter to $0$, 
\[
K_0^V(0)=\kappa_0\pi(T_0^{-1}),\qquad R_i^V(0)=\kappa\pi(T_i^{-1}),\qquad K_n^V(0)=\kappa_n\pi(T_n^{-1})
\]
for $1\leq i<n$.
When specializing the spectral parameter to $1$ we recover the identity:
\[ \begin{split} K_0^V(1) &= \Id_V =  K_n^V(1), \\
R_i^V(1) &= \Id_V, \qquad 1 \leq i < n. \end{split} \]
%%%%%%%%%%%%%%%%%%%%%%%%%%%%%%%%%%%%%%%%%%%%%%%%%%%%%%
\end{prop}
We suppress the dependence on $V$ in the notations if $V$ is the spin representation $\bigl(\rho^{\underline{\kappa}}_{\psi_0,\psi_n},\bigl(\mathbb{C}^2\bigr)^{\otimes n}\bigr)$.
Proposition \ref{Baxterization} extends
results of Cherednik (see, e.g., \cite[\S 1.3.2]{CBook} and references therein) 
to the "Koornwinder case" (also known as the $C^\vee C$-case). 

Before sketching the proof of Proposition \ref{Baxterization}, we first show that it produces 
Baxterizations of the solution $\Upsilon$ of the quantum Yang-Baxter equation and of the solutions
$\overline{K}$ and $K$ of the associated reflection equations when it is applied to the spin representations
of the affine Hecke algebra $H(\underline{\kappa})$
(see Subsection \ref{Spinrep} for the definitions of $\Upsilon,\overline{K}$ and $K$). Recall that $P:\mathbb{C}^2\otimes\mathbb{C}^2
\rightarrow\mathbb{C}^2\otimes\mathbb{C}^2$ is the flip operator.
%%%%%%%%%%%%%%%%%%%%%%%%%%%%%%%%%%%%%%%%%%%%%%%%%%%%%%%
\begin{cor}\label{rk}
For parameters $\underline{\kappa}$, $\upsilon_0,\upsilon_n, \psi_0$ and $\psi_n$ consider the matrices 
\begin{align*} 
\bar k(x) &:= \frac{\kappa_0}{(1-\kappa_0 \upsilon_0 x)(1+ \kappa_0 \upsilon_0^{-1} x)} \left( \begin{smallmatrix} (\kappa_0^{-1}-\kappa_0) x^2 + (\upsilon_0^{-1} - \upsilon_0) x & \psi_0 (1-x^2) \\ \psi_0^{-1}(1-x^2) & \kappa_0^{-1} -\kappa_0 + (\upsilon_0^{-1} - \upsilon_0) x \end{smallmatrix} \right), \\
r(x) &:=  \frac{1}{1-\kappa^2 x} \left( \begin{smallmatrix} 1 - \kappa^2 x & 0 & 0 & 0 \\ 0 & \kappa(1-x) & 1-\kappa^2 & 0 \\ 
0 & (1-\kappa^2)x & \kappa (1-x) & 0 \\ 0 & 0 & 0 & 1 - \kappa^2 x \end{smallmatrix} \right), \\
k(x) &:=  \frac{\kappa_n}{(1-\kappa_n \upsilon_n x)(1+ \kappa_n \upsilon_n^{-1} x)} \left( \begin{smallmatrix} \kappa_n^{-1}-\kappa_n + (\upsilon_n^{-1} - \upsilon_n) x & \psi_n^{-1} (1-x^2) \\  \psi_n(1-x^2) & (\kappa_n^{-1}-\kappa_n)x^2 + (\upsilon_n^{-1} - \upsilon_n) x \end{smallmatrix} \right). \end{align*}
Then 
\begin{enumerate}
\item[{\bf 1.}] $r(x)$ satisfies the quantum Yang-Baxter equation with spectral parameter
\begin{equation} \label{YBEnew}
r_{12}(x) r_{13}(xy) r_{23}(y) = r_{23}(y) r_{13}(xy) r_{12}(x)
\end{equation}
as endomorphisms of $\mathbb{C}^2\otimes\mathbb{C}^2\otimes\mathbb{C}^2$. Furthermore, 
$r(0)=\kappa\Upsilon^{-1}$ (Baxterization), $r(x)r_{21}(x^{-1})=\textup{Id}_{\mathbb{C}^2\otimes\mathbb{C}^2}$ (unitarity) and 
$r(1)=P$ (regularity).
\item[{\bf 2.}] $\overline{k}(x)$ satisfies the reflection equation with spectral parameter
\begin{equation}\label{REleftnew}
r_{12}(x/y) \bar k_2(x) r_{21}(xy) \bar k_1(y) = \bar k_1(y) r_{12}(xy) \bar k_2(x) r_{21}(x/y) 
\end{equation}
as endomorphisms of $\mathbb{C}^2\otimes\mathbb{C}^2$. 
Furthermore, $\overline{k}(0)=\kappa_0\overline{K}^{-1}$ (Baxterization),
$\overline{k}(x)\overline{k}(x^{-1})=\textup{Id}_{\mathbb{C}^2}$ (unitarity) and 
$\overline{k}(1)=\textup{Id}_{\mathbb{C}^2}=\overline{k}(-1)$ (regularity).
\item[{\bf 3.}] $k(x)$ satisfies the reflection equation with spectral parameter
\begin{equation}\label{RErightnew}
r_{12}(x/y) k_1(x) r_{21}(xy) k_2(y) = k_2(y) r_{12}(xy) k_1(x) r_{21}(x/y)
\end{equation}
as endomorphisms of $\mathbb{C}^2\otimes\mathbb{C}^2$. 
Furthermore, $k(0)=\kappa_nK^{-1}$ (Baxterization),
$k(x)k(x^{-1})=\textup{Id}_{\mathbb{C}^2}$ (unitarity) and 
$k(1)=\textup{Id}_{\mathbb{C}^2}=k(-1)$ (regularity).
\end{enumerate}
\end{cor}
%%%%%%%%%%%%%%%%%%%%%%%%%%%%%%%%%%%%%%%%%%%%%%%%%%%%%
\begin{proof}
Apply Proposition \ref{Baxterization} to the $H(\underline{\kappa})$-module $\bigl(\rho^{\underline{\kappa}}_{\psi_0,\psi_n},\bigl(\mathbb{C}^2\bigr)^{\otimes n}\bigr)$ and use that the associated linear operators $K_0(x), R_i(x) $ and $K_n(x)$ on $\bigl(\mathbb{C}^2\bigr)^{\otimes n}$ are given by $\overline{k}_1(x), (r(x)\circ P)_{i\,i+1}$ and $k_n(x)$, respectively ($1\leq i<n$).
\end{proof}
%%%%%%%%%%%%%%%%%%%%%%%%%%%%%%%%%%%%%%%%%%%%%%%%%%%%%%
\begin{rema}
The solution $r(x)$ of the quantum Yang-Baxter equation is gauge-equivalent to the $R$-matrix of the XXZ spin-$\frac{1}{2}$ chain. 
Corollary \ref{rk} thus gives rise to a three-parameter family of solutions of the associated reflection equation (with the free parameters being $\kappa_0,\upsilon_0,\psi_0$). 
This three-parameter family of solutions of the reflection equation was written down before in \cite{dVGR} (also see \cite{Ne}).
\end{rema}
%%%%%%%%%%%%%%%%%%%%%%%%%%%%%%%%%%%%%%%%%%%%%%%%%%%%%%%

The proof of Proposition \ref{Baxterization} uses Noumi-Sahi's \cite{N,Sa} extension of Cherednik's \cite{CBook}
double affine Hecke algebra and its basic representation, notions that also play a central role in the theory on Koornwinder polynomials; see \cite{N,Sa}
and Subsection \ref{Kosection}. We sketch now some of the key steps of the proof of Proposition \ref{Baxterization}.

Fix $q\in\mathbb{C}^*$ and a square root $q^{\frac{1}{2}}$ once and for all. 
A $q$-dependent action of the affine Weyl group $W=W_0\ltimes\mathbb{Z}^n$ on $T=\bigl(\mathbb{C}^*\bigr)^n$ 
is given by
\begin{equation*}
\begin{split}
s_0\bm{t}&:=(qt_1^{-1},t_2,\ldots,t_n),\\
s_i\bm{t}&:=(t_1,\ldots,t_{i-1},t_{i+1},t_i,t_{i+2},\ldots,t_n),\qquad 1\leq i<n,\\
s_n\bm{t}&:=(t_1,\ldots,t_{n-1},t_n^{-1}).
\end{split}
\end{equation*}
Note that $\tau( \mu)\bm{t}=q^{ \mu}\bm{t}:=(q^{\mu_1}t_1,\ldots,q^{\mu_n}t_n)$ for $ \mu = (\mu_1,\ldots,\mu_n) \in \mathbb{Z}^n$.
By transposition, $W$ acts on the fields $\mathbb{C}(T)$ and $\mathcal{M}(T)$ of rational and meromorphic functions on $T$
by field automorphisms.
The following lemma is equivalent to Proposition \ref{Baxterization} because of the Coxeter presentation of $W$ with
respect to the simple reflections $s_0,\ldots,s_n$.
%%%%%%%%%%%%%%%%%%%%%%%%%%%%%%%%%%%%%%%%%%%%%%%%%%%%%%%%
\begin{lem}\label{reform}
Let $\pi: H(\underline{\kappa})\rightarrow\End_{\mathbb{C}}(V)$ be a representation.
There exist unique $C_w^V=C_w^V(\cdot;\underline{\kappa};\upsilon_0,\upsilon_n)\in\mathbb{C}(T)\otimes\End_{\mathbb{C}}(V)$
($w\in W$) satisfying the cocycle conditions
\begin{align}
C_e^V(\bm{t})&=\Id_V,\\
\label{cocycle} C_{ww^\prime}^V(\bm{t})&=C_w^V(\bm{t})C_{w^\prime}^V(w^{-1}\bm{t}),\qquad \forall\, w,w^\prime\in W
\end{align}
(with $e\in W$ the unit element) and satisfying
\[
C_{s_0}^V(\bm{t})=K_0^V(q^{\frac{1}{2}}/t_1),\qquad
C_{s_i}^V(\bm{t})=R_i^V(t_i/t_{i+1}),\qquad
C_{s_n}^V(\bm{t})=K_n^V(t_n)
\]
for $1\leq i<n$.
\end{lem}
%%%%%%%%%%%%%%%%%%%%%%%%%%%%%%%%%%%%%%%%%%%%%%%%%%%%%%%%%
Define the algebra $\mathbb{C}(T)\#W$ of $q$-difference reflection operators with rational coefficients
as the vector space $\mathbb{C}(T)\otimes \mathbb{C}[W]$ with the multiplication rule
\[
(p\otimes w)(r\otimes w^\prime):=p(w\cdot r)\otimes ww^\prime,\qquad p,r\in\mathbb{C}(T),\,\,\, w,w^\prime\in W,
\]
where $w\cdot r$ is the result of the $q$-dependent action of $w\in W$ on $r\in \mathbb{C}(T)$. We simply write $pw$ for $p\otimes w\in\mathbb{C}(T)\#W$.
The key of the proof of Lemma \ref{reform} is to write the simple reflections $s_i \in\mathbb{C}(T)\otimes\mathbb{C}[W]$ ($0\leq i\leq n$)
in terms of Noumi's \cite{N} two-parameter family of realizations of the affine Hecke algebra $H(\underline{\kappa})$ as a subalgebra of $\mathbb{C}(T)\#W$,
which is given as follows.

Define $c_i=c_i(\cdot;\underline{\kappa};\upsilon_0,\upsilon_n)\in\mathbb{C}(T)$ for $0\leq i\leq n$
by the explicit formulas
\begin{equation*}
\begin{split}
c_0(\bm{t})&:=\kappa_0^{-1}\frac{(1-q^{\frac{1}{2}}\kappa_0\upsilon_0t_1^{-1})(1+q^{\frac{1}{2}}\kappa_0\upsilon_0^{-1}t_1^{-1})}
{(1-qt_1^{-2})},\\
c_i(\bm{t})&:=\kappa^{-1}\frac{(1-\kappa^2t_i/t_{i+1})}{(1-t_i/t_{i+1})},\\
c_n(\bm{t})&:=\kappa_n^{-1}\frac{(1-\kappa_n\upsilon_nt_n)(1+\kappa_n\upsilon_n^{-1}t_n)}{(1-t_n^2)}
\end{split}
\end{equation*}
for $1\leq i<n$. 
%%%%%%%%%%%%%%%%%%%%%%%%%%%%%%%%%%%%%%%%%%%%%%%%%%%%%%%%%%%%%
\begin{thm}[Noumi] \label{thmN}
Let $\upsilon_0,\upsilon_n\in\mathbb{C}^*$. There exists a unique unital algebra embedding 
\[
\iota=\iota_{\upsilon_0,\upsilon_n}^{\underline{\kappa}}:H(\underline{\kappa})\hookrightarrow \mathbb{C}(T)\#W
\]
satisfying 
\[
\iota(T_i)=\kappa_i+c_i(s_i-e),\qquad 0\leq i\leq n.
\]
\end{thm}
%%%%%%%%%%%%%%%%%%%%%%%%%%%%%%%%%%%%%%%%%%%%%%%%%%%%%%%%%%%%%
\begin{rema}
The Noumi-Sahi extension of Cherednik's double affine Hecke algebra is the subalgebra of $\mathbb{C}(T)\#W$ generated by
$\iota_{\upsilon_0,\upsilon_n}^{\underline{\kappa}}(H(\underline{\kappa}))$ and by the algebra $\mathbb{C}[T]$ of regular functions on $T$.
\end{rema}
%%%%%%%%%%%%%%%%%%%%%%%%%%%%%%%%%%%%%%%%%%%%%%%%%%%%%%%%%%%%%%
The following consequence should be compared to \cite[Prop. 3.5]{St1}.
%%%%%%%%%%%%%%%%%%%%%%%%%%%%%%%%%%%%%%%%%%%%%%%%%%%%%%%%%%%%%%
\begin{cor}\label{corCrit}
Let $\pi: H(\underline{\kappa})\rightarrow \End_{\mathbb{C}}(V)$ be a representation.
There exists a unique algebra homomorphism 
\[
\nabla^V: \mathbb{C}(T)\#W\rightarrow \mathbb{C}(T)\#W\otimes\End_{\mathbb{C}}(V)
\]
satisfying $\nabla^V(p)=p\otimes\Id_V$ for $p\in\mathbb{C}(T)$ and
\[
\nabla^V(\iota(T_i))=s_i\otimes \pi(T_i)+(c_i-\kappa_i)(s_i-e)\otimes \Id_V,\qquad 0\leq i\leq n.
\]
\end{cor}
%%%%%%%%%%%%%%%%%%%%%%%%%%%%%%%%%%%%%%%%%%%%%%%%%%
\begin{proof}
Consider the induced $\mathbb{C}(T)\#W$-module
\[\textup{Ind}_{H(\underline{\kappa})}^{\mathbb{C}(T)\#W}\bigl(V\bigr)=\mathbb{C}(T)\#W\otimes_{H(\underline{\kappa})}V,
\]
where we identify $H(\underline{\kappa})$ with the subalgebra $\iota(H(\underline{\kappa}))$ of $\mathbb{C}(T)\#W$.
Using the linear isomorphism 
\[\mathbb{C}(T)\otimes \iota(H(\underline{\kappa}))\overset{\sim}{\longrightarrow}\mathbb{C}(T)\#W
\]
defined by the multiplication map, the representation map $\nabla$
of $\textup{Ind}_{H(\underline{\kappa})}^{\mathbb{C}(T)\#W}\bigl(V\bigr)$ becomes an algebra map
\[
\nabla: \mathbb{C}(T)\#W\rightarrow \End_{\mathbb{C}}(\mathbb{C}(T)\otimes_{\mathbb{C}}V).
\]
View $\mathbb{C}(T)\#W\otimes_{\mathbb{C}}\pi(H(\underline{\kappa}))$ as a subalgebra of $\End_{\mathbb{C}}(\mathbb{C}(T)\otimes_{\mathbb{C}}V)$,
with $\mathbb{C}(T)\#W$ acting on $\mathbb{C}(T)$ as $q$-difference reflection operators. It then suffices to show that  $\nabla(p)=\nabla^V(p)$ and $\nabla(T_i)=\nabla^V(T_i)$ for $p\in\mathbb{C}(T)$ and $0\leq i\leq n$. This follows by a direct computation
using the commutation relations 
\[
T_ip=(s_i\cdot p)T_i+(\kappa_i-c_i)(p-s_i\cdot p),\qquad 0\leq i\leq n,\,\,\, p\in\mathbb{C}(T)
\]
in $\mathbb{C}(T)\#W$.
\end{proof}
%%%%%%%%%%%%%%%%%%%%%%%%%%%%%%%%%%%%%%%%%%%%%%%%%%%%
Observe that
\begin{equation*}
\begin{split}
\nabla^V(s_i)&=\nabla^V(c_i^{-1}(T_i-\kappa_i+c_i))\\
&=c_i^{-1}s_i\otimes\pi(T_i)+c_i^{-1}(c_i-\kappa_i)(s_i-e)\otimes\textup{Id}_V+c_i^{-1}(c_i-\kappa_i)e\otimes\textup{Id}_V\\
&=(c_i^{-1}e\otimes\pi(T_i)+c_i^{-1}(c_i-\kappa_i)e\otimes\textup{Id}_V)(s_i\otimes\textup{Id}_V)\\
&=C_{s_i}^V(\cdot)(s_i\otimes\Id_V)
\end{split}
\end{equation*}
for $i=0,\ldots,n$, where the last equality follows from the fact that
\begin{equation*}
\begin{split}
\frac{\pi(T_0)+c_0(\bm{t})-\kappa_0}{c_0(\bm{t})}&=K_0^V(q^{\frac{1}{2}}t_1^{-1})=C_{s_0}^V(\bm{t}),\\
\frac{\pi(T_i)+c_i(\bm{t})-\kappa}{c_i(\bm{t})}&=R_i^V(t_i/t_{i+1})=C_{s_i}^V(\bm{t}),\qquad 1\leq i<n,\\
\frac{\pi(T_n)+c_n(\bm{t})-\kappa_n}{c_n(\bm{t})}&=K_n^V(t_n)=C_{s_n}^V(\bm{t}),
\end{split}
\end{equation*}
which in turn can be checked by a direct computation.
Combined with Corollary \ref{corCrit} this proves Lemma \ref{reform}. 
Finally note that 
\[
\nabla^V(w)=C_w^V(\cdot)(w\otimes\Id_V),\qquad w\in W.
\]
%%%%%%%%%%%%%%%%%%%%%%%%%%%%%%%%%%%%%%%%%%%%%%%%%%%%%%%%%%%%%%%%%
\subsection{The reflection quantum KZ equations} \label{ReflqKZeqns}
%%%%%%%%%%%%%%%%%%%%%%%%%%%%%%%%%%%%%%%%%%%%%%%%%%%%%%%%%%%%%%%%%%

Let $\mathcal{M}(T)$ be the field of meromorphic functions on $T=\bigl(\mathbb{C}^*\bigr)^n$. 
%%%%%%%%%%%%%%%%%%%%%%%%%%%%%%%%%%%%%%%%%%%%%%%%%
\begin{defi}
Let $V$ be a $H(\underline{\kappa})$-module and $\upsilon_0,\upsilon_n\in\mathbb{C}^*$. We say that $f\in\mathcal{M}(T)\otimes V$ is a solution of the associated
reflection quantum Knizhnik-Zamolodchikov (KZ) equations if
\[
\nabla^V(\tau(\lambda))f=f\qquad \forall\,\, \lambda\in\mathbb{Z}^n.
\] 
We write $\textup{Sol}_{KZ}(V)=\textup{Sol}_{KZ}(V;\underline{\kappa};\upsilon_0,\upsilon_n)\subseteq \mathcal{M}(T)\otimes V$ for the associated space of solutions of the reflection quantum KZ equations.
\end{defi}
%%%%%%%%%%%%%%%%%%%%%%%%%%%%%%%%%%%%%%%%%%%%%%%%%%
Note that for $\lambda\in\mathbb{Z}^n$,
\[
\bigl(\nabla^V(\tau(\lambda))f\bigr)(\bm{t})=C_{\tau(\lambda)}(\bm{t})f(q^{-\lambda}\bm{t})
\]
with transport operators $C_{\tau(\lambda)}\in \mathbb{C}(T)\otimes\textup{End}_{\mathbb{C}}(V)$. 
Consequently the reflection quantum KZ equations form a consistent system of first order, linear $q$-difference equations. Note also that $\textup{Sol}_{KZ}(V)$ is a left $W_0$-module with the action given
by
\[
\bigl(w\cdot f\bigr)(\bm{t}):=\bigl(\nabla^V(w)f\bigr)(\bm{t})=C_w(\bm{t})f(w^{-1}\bm{t}),\qquad w\in W_0.
\]

The reflection quantum KZ equations are equivalent to
\[
C_{\tau_i}^V(\bm{t})f(q^{-\epsilon_i}\bm{t})=f(\bm{t}) \qquad \forall\, i=1,\ldots,n,
\]
where $\{\epsilon_i\}_{i=1}^n$ is the standard orthonormal basis of $\mathbb{R}^n$.
Using the expression \eqref{taui} of $\tau_i\in W$ as product of simple reflections and the cocycle property \eqref{cocycle}, $C_{\tau_i}^V$ takes on the explicit form
\begin{equation}\label{explicitC}
\begin{split}
C_{\tau_i}^V&(\bm{t})=R_{i-1}^V(t_{i-1}/t_i)R_{i-2}^V(t_{i-2}/t_i)\cdots R_1^V(t_1/t_i)K_0^V(q^{\frac{1}{2}}/t_i)\\
&\times R_1^V(q/t_1t_i)R_2^V(q/t_2t_i)\cdots R_{i-1}^V(q/t_{i-1}t_i)R_i^V(q/t_it_{i+1})\cdots R_{n-1}^V(q/t_it_n)\\
&\times K_n^V(q/t_i)R_{n-1}^V(qt_n/t_i)\cdots R_i^V(qt_{i+1}/t_i).
\end{split}
\end{equation}

%%%%%%%%%%%%%%%%%%%%%%%%%%%%%%%%%%%%%%%%%%%%%%%%%%%%%%%%%
\begin{rema}
{\bf (i)} In Section \ref{Solsection} we construct for principal series representations 
solutions of the associated reflection quantum KZ equations in terms of nonsymmetric Koornwinder polynomials.
As a special case this construction gives rise to solutions for the spin representations 
$\bigl(\rho^{\underline{\kappa}}_{\psi_0,\psi_n},\bigl(\mathbb{C}^2\bigr)^{\otimes n}\bigr)$ of $H(\underline{\kappa})$
in view of Proposition \ref{rhopiequivalence}. 
Alternatively, for special classes of parameters solutions of the reflection quantum KZ equations for spin representations can be constructed using a vertex operator approach (see \cite{JKKKM,JKKMW,We} for such a treatment involving diagonal $K$-matrices, and \cite{BKojima} for a recent extension to triangular $K$-matrices) or by a generalized Bethe ansatz method \cite{RSV} (for diagonal $K$-matrices).\\
{\bf (ii)} Write $\textup{Sol}_{KZ}^{\underline{\kappa}}(V)^{W_0}$ for the subspace of $W_0$-invariant solutions of the reflection quantum KZ equations.
Then $f\in\textup{Sol}_{KZ}^{\underline{\kappa}}(V)^{W_0}$ if and only if
\begin{equation*}
\begin{split}
K_0^V(q^{\frac{1}{2}}/t_1)f(s_0\bm{t})&=f(\bm{t}),\\
R_i^V(t_i/t_{i+1})f(s_i\bm{t})&=f(\bm{t}),\qquad 1\leq i<n,\\
K_n^V(t_n)f(s_n\bm{t})&=f(\bm{t}).
\end{split}
\end{equation*}
It is in this form that the reflection quantum KZ equations often appear in the literature on spin chains, see, e.g., 
\cite{dFZJ, dGPS, dGP2,PoThesis, ZJ2}.
\end{rema}
%%%%%%%%%%%%%%%%%%%%%%%%%%%%%%%%%%%%%%%%%%%%%%%%%%%%%%%%

%%%%%%%%%%%%%%%%%%%%%%%%%%%%%%%%%%%%%%%%%%%%%%%%%%%%%%%%
\subsection{The XXZ spin chain} \label{spsection}
%%%%%%%%%%%%%%%%%%%%%%%%%%%%%%%%%%%%%%%%%%%%%%%%%%%%%%%%%

One of the main aims in the study of spin chain models
is to find the spectrum of the quantum Hamiltonian, a distinguished operator acting on the state space in which the combined states of $n$ individual spins reside, and to find a complete set of eigenfunctions. 
The method of commuting transfer operators, pioneered by Baxter (see \cite{Ba} and references therein) and elaborated upon principally by the Faddeev school (e.g. \cite{Fa,Sk1982,KBI}), produces a generating function of quantum Hamiltonians from the basic data of the models ($R$- and $K$-matrices). 
In many cases the basic data arise from the representation theory of quantum groups or Hecke algebras.

We consider now briefly the construction of the quantum integrable model associated to the basic data $r(x)$, $k(x)$ and $\overline{k}(x)$, see Corollary \ref{rk} (in this case the basic data thus are obtained by a Baxterization procedure from the spin representation $\hat{\rho}^{\underline{\kappa}}_{\psi_0,\psi_n}$).
The associated state space is the representation space $(\C^2)^{\otimes n}$ of the spin representation.
To write down the associated transfer operators we first consider a larger space, which for the pertinent case is $\underset{(0)}{\C^2}  \otimes (\C^2)^{\otimes n} $, where the additional copy of $\C^2$ numbered 0 is referred to as the auxiliary space. 
Given a spectral parameter $x \in \C^*$ and an $n$-tuple of so-called inhomogeneities $\bm t = (t_1,\ldots,t_n) \in T=(\C^*)^n$, one constructs from the matrices $\check r(x) = r(x) \circ P$ and $k(x)$ the \emph{monodromy operator}
\begin{align*}
U_0(x;\bm t) &= \check r_{01}(x t_1^{-1}) \check r_{12}(x t_2^{-1}) \cdots \check r_{n\!-\!1 \, n}(x t_n^{-1}) k_n(x)  \check r_{n\!-\!1 \, n}(x t_n) \cdots \check r_{12}(x t_2) \check r_{01}(x t_1)  \\
&= r_{01}(x t_1^{-1}) \cdots r_{0n}(x t_n^{-1}) k_0(x) r_{n0}(x t_n) \cdots r_{10}(x t_1) ,
\end{align*}
acting on $\underset{(0)}{\C^2}\otimes\bigl(\mathbb{C}^2\bigr)^n=
\underset{(0)}{\C^2} \otimes \underset{(1)}{\C^2} \otimes \cdots \otimes \underset{(n)}{\C^2}$ (with the sublabels indicating the numbering of the tensor legs on which the
operator is acting). 

Fix a square root $\kappa^{\frac{1}{2}}$ of $\kappa$ and set
\[
\theta:=\left(\begin{matrix} \kappa^{-\frac{1}{2}} & 0\\ 0 & \kappa^{\frac{1}{2}}\end{matrix}\right).
\]
The \emph{transfer matrix} is the endomorphism of $\bigl(\mathbb{C}^2\bigr)^n$ given by
\[ T(x;\bm t) = \Tr_0 \bigl(\theta_0 \, \bar k_0(\kappa^2 x) \theta_0 U_0(x;\bm t)\bigr),
\]
where $\Tr_0$ stands for the partial trace over the auxiliary space.
The $\kappa^2$-shift in the argument of $\bar k$ is absent from the standard approach in \cite{Sk1988,MN,Ne}. 
It appears here to allow us to relate the transfer operators to the transport operators $C_{\tau_i}(\bm t)$, see Proposition \ref{transferoperatorandcocycles}.
The following result is the well-known statement that transfer operators are generating functions of commuting quantum Hamiltonians.
%%%%%%%%%%%%%%%%%%%%%%%%%%%%%%%%%%%%%%%%%%%%%%%%%%%%%%
\begin{thm} \label{commutingtransferoperators}
We have
\begin{equation} \label{eqn:commutingtransferoperators} [T(x;\bm t),T(y;\bm t)]=0. \end{equation}
as endomorphisms of $\bigl(\mathbb{C}^2\bigr)^{\otimes n}$ depending rationally on $\mathbf{t}\in T$.
\end{thm}
%%%%%%%%%%%%%%%%%%%%%%%%%%%%%%%%%%%%%%%%%%%%%%%%%%%%%%%
\begin{proof}
The proof is a straightforward modification of Sklyanin's \cite{Sk1988} proof for $P$- and $T$-symmetric $R$-matrices; note that the presence of the inhomogeneities does not affect the argument at all.
We remark that $r$ satisfies PT-symmetry, $r_{21}(x) = r^T_{12}(x)$, and crossing unitarity:
\[ \theta^{-1}_1 \theta^{-1}_2  r^{T_1}_{12}(\kappa^{-4}x^{-1}) \theta_1 \theta_2 r^{T_2}_{12}(x) = \Phi(x) \Id_{\C^2\otimes \C^2}, \qquad \Phi(x) =  \frac{(1-x)(1-\kappa^4 x)}{(1-\kappa^2 x)^2}\]
where $T_i$ denotes the partial transpose with respect to the tensor leg labelled $i$. 
Crucially, note that \eqref{REleftnew}, the reflection equation for $\bar k$, can be re-written as
\begin{equation}\label{reformRE} \begin{gathered}
r_{12}(y/x) \bar k^{T_1}_1(\kappa^2 x) r_{21}((\kappa^{4}xy)^{-1}) \bar k^{T_2}_2(\kappa^2 y) = \hspace{30mm} \\
\hspace{30mm}  = \bar k^{T_2}_2(\kappa^2 y) r_{12}((\kappa^{4}xy)^{-1}) \bar k^{T_1}_1(\kappa^2 x) r_{21}(y/x), \end{gathered} \end{equation}
which is equivalent to \cite[Eqn. (10)]{MN}. 
To obtain this reformulation of \eqref{REleftnew}, we have used that $r^T_{12}(x) = r_{12}(x^{-1})|_{\kappa \to \kappa^{-1}}$ and $\bar k(x)$ does not depend on $\kappa$, and that $\bar k^T(x) = \bar k(x)|_{\psi_0 \to \psi_0^{-1}}$ and $r(x)$ does not depend on $\psi_0$.
Using \eqref{reformRE} the proof is essentially the one implicitly present in \cite{MN}.
\end{proof}
%%%%%%%%%%%%%%%%%%%%%%%%%%%%%%%%%%%%%%%%%%%%%%%%%%%%%%%%%

The transport operators $C_{\tau_i}(\bm t)$ of the reflection quantum KZ equation associated to the spin representation $\widehat{\rho}^{\underline{\kappa}}_{\psi_0,\psi_n}$
are linear operators on $\bigl(\mathbb{C}^2\bigr)^n$ depending rationally on $\bm t\in T$. They are explicitly expressed in terms of $r(x)$, $k(x)$ and $\overline{k}(x)$ by
\begin{equation}\label{explicitCrho}
\begin{split}
&C_{\tau_i}(\bm t)=\check{r}_{i-1\,i}(t_{i-1}/t_i)\check{r}_{i-2\,i-1}(t_{i-2}/t_{i-1})\cdots \check{r}_{12}(t_1/t_i)\\
&\,\,\times \overline{k}_1(q^{\frac{1}{2}}/t_i)\check{r}_{12}(q/t_1t_i)\cdots\check{r}_{i-1\,i}(q/t_{i-1}t_i)\\
&\,\,\times \check{r}_{i\,i+1}(q/t_it_{i+1})\cdots\check{r}_{n-1\,n}(q/t_it_n)k_n(q/t_i)
\check{r}_{n-1\,n}(qt_n/t_i)\cdots\check{r}_{i\,i+1}(qt_{i+1}/t_i)
\end{split}
\end{equation}
in view of \eqref{explicitC} and the proof of Corollary \ref{rk}. Note that the cocycle property of $C_w$ $(w\in W)$ implies that
\[
C_{\tau_i}(\bm t)C_{\tau_j}(q^{-\epsilon_i}\bm t)=C_{\tau_j}(\bm t)C_{\tau_i}(q^{-\epsilon_j}\bm t),\qquad 1\leq i,j\leq n,
\]
where $\{\epsilon_i\}_i$ denotes the standard orthonormal basis of $\mathbb{R}^n$.
To relate the transport operators to the transfer operator
we need the \emph{boundary crossing symmetry} of the $K$-matrix $\overline{k}(x)$,
\begin{equation} \label{reflectioncrossingsymmetry}
\Tr_0
\bigl(
\theta_0 \bar k_0(\kappa^2 x) \theta_0 \check r_{01}(x^2)
\bigr) 
= \Phi_\mathrm{bdy}(x)  \bar k_1(x)
\end{equation}
as linear operators on $\underset{(1)}{\mathbb{C}^2}$,
where 
\[
\Phi_{\mathrm{bdy}}(x):=\kappa\frac{(1-\kappa_0\upsilon_0x)(1+\kappa_0\upsilon_0^{-1}x)(1-\kappa^4x^2)}
{(1-\kappa^2\kappa_0\upsilon_0x)(1+\kappa^2\kappa_0\upsilon_0^{-1}x)(1-\kappa^2x^2)}.
\]
See \cite{GZ} for a discussion of the notion of boundary crossing symmetries.
%%%%%%%%%%%%%%%%%%%%%%%%%%%%%%%%%%%%%%%%%%%%%%%%%%%%%%%%%%%%%%%%%
\begin{prop} \label{transferoperatorandcocycles}
For $i=1,\ldots,n$ we have
\begin{equation}\label{interpolants}
\begin{split}
T(t_i^{-1};\bm t)&=\Phi_\mathrm{bdy}(t_i^{-1})C_{\tau_i}(\bm t)|_{q=1},\\
T(t_i;\bm t)&=\Phi_\mathrm{bdy}(t_i)C_{\tau_i}(\bm t)^{-1}|_{q=1}.
\end{split}
\end{equation}
\end{prop}
%%%%%%%%%%%%%%%%%%%%%%%%%%%%%%%%%%%%%%%%%%%%%%%%%%%%%%%%%%%%%%
\begin{proof}
The transfer operator satisfies
\begin{equation*}
\begin{split}
T(x;\bm t) \check r_{i \, i+1}(t_i/t_{i+1}) &= \check r_{i \, i+1}(t_i/t_{i+1})  T(x;s_i \bm t), \qquad 1 \leq i < n, \\
T(x;\bm t) k_n(t_n) &= k_n(t_n)  T(x;s_n \bm t).
\end{split}
\end{equation*}
On the other hand, the transport operators satisfy
\begin{equation*}
\begin{split}
C_{\tau_{i+1}}(\bm t)&=\check{r}_{i\,i+1}(t_i/t_{i+1})C_{\tau_i}(\bm t)\check{r}_{i\,i+1}(t_{i+1}/qt_i),\qquad 1\leq i<n,\\
C_{\tau_n}(\bm t)^{-1}&=k_n(t_n/q)C_{\tau_n}(s_n\tau_n^{-1}\bm t)k_n(t_n^{-1}).
\end{split}
\end{equation*}
Hence it suffices to prove the first equality of \eqref{interpolants} for $i=1$.
Using the regularity $\check{r}(1)=\Id_{\mathbb{C}^2\otimes\mathbb{C}^2}$ of the $R$-matrix the desired equality
\[
T(t_1^{-1};\bm t)=\Phi_\mathrm{bdy}(t_1^{-1})C_{\tau_1}(\bm t)|_{q=1}
 \]
follows by a direct computation using the definition of $T(x;\bm t)$, the boundary crossing symmetry \eqref{reflectioncrossingsymmetry}, and \eqref{explicitCrho}.
\end{proof}
%%%%%%%%%%%%%%%%%%%%%%%%%%%%%%%%%%%%%%%%%%%%%%%%%

%%%%%%%%%%%%%%%%%%%%%%%%%%%%%%%%%%%%%%%%%%%%%%%%%%
\begin{rema}
A detailed study of the relation between transfer operators and transport operators in the context of reflection quantum KZ equations is part of ongoing joint work with N. Reshetikhin.
\end{rema}
%%%%%%%%%%%%%%%%%%%%%%%%%%%%%%%%%%%%%%%%%%%%%%%%%%%

%%%%%%%%%%%%%%%%%%%%%%%%%%%%%%%%%%%%%%%%%%%%%%%%%%%
\subsection{The Hamiltonian}\label{Hsection}
%%%%%%%%%%%%%%%%%%%%%%%%%%%%%%%%%%%%%%%%%%%%%%%%%%%

Write $\bm 1 = (1,\ldots,1) \in T$ and $'$ for the derivative with respect to the spectral parameter. We suppose in this subsection that the 
parameters are generic.
The following definition corresponds to \cite[Eqn. (25)]{MN}.
%%%%%%%%%%%%%%%%%%%%%%%%%%%%%%%%%%%%%%%%%%%%%%%%%%%%%%%%%%%%%%%%%%%
\begin{defi}
The Hamiltonian for the XXZ Heisenberg spin chain with general boundary conditions is defined as
\begin{align*}
H^\text{XXZ}_\text{bdy} &= \frac{\kappa-\kappa^{-1}}{2} \frac{\mathrm{d}}{\mathrm{d}x} \log\Bigl( \frac{T(x;\bm 1)}{\Tr \bigl( \theta \bar k(\kappa^2 x) \theta \bigr)}\Bigr)\Bigr|_{x=1}- C_0 \Id_{(\C^2)^{\otimes n}}  \\
&= (\kappa-\kappa^{-1}) \biggl( \sum_{i=1}^{n-1} \check{r}'_{i \, i \! + \! 1}(1) + \frac{ k'_n(1)}{2}  + \frac{\Tr_0\bigl(\theta_0 \bar k_0(\kappa^2) \theta_0  \check{r}'_{01}(1)\bigr)}{\Tr\bigl(\theta \bar k(\kappa^2) \theta\bigr)}   \biggr) - C_0 \Id_{(\C^2)^{\otimes n}}
\end{align*}
where
\[ C_0 = \frac{-1}{\kappa(1+\kappa^2)} \frac{(\kappa^2-\kappa_0 \upsilon_0)(\kappa^2 + \kappa_0 \upsilon_0^{-1})}{(1-\kappa_0 \upsilon_0)(1+\kappa_0 \upsilon_0^{-1})}. \]
\end{defi}
%%%%%%%%%%%%%%%%%%%%%%%%%%%%%%%%%%%%%%%%%%%%%%%%%%%%%%%%%%%%%%%%%%%
\begin{rema}
To find the spectrum of the Hamiltonian, one can now use the fact that, by construction, it commutes with all $T(x;\bm 1)$. 
It is therefore sufficient to find a complete set of common eigenfunctions of the $T(x;\bm 1)$, for which one typically uses the algebraic Bethe ansatz and related methods, which for \emph{diagonal} boundary conditions was first done by Sklyanin \cite{Sk1988}. 
For the non-diagonal case, it may be possible to use a variant of the algebraic Bethe ansatz involving so-called dynamical $R$-and $K$-matrices (see \cite{FK} for such a treatment in a special case).
\end{rema}
%%%%%%%%%%%%%%%%%%%%%%%%%%%%%%%%%%%%%%%%%%%%%%%%%%%%%%%%%%%%%%%%

Recall the Pauli spin matrices
\[ \sigma^X = \begin{pmatrix} 0 & 1 \\ 1 & 0 \end{pmatrix}, \qquad
 \sigma^Y = \begin{pmatrix} 0 & -\sqrt{-1} \\ \sqrt{-1} & 0 \end{pmatrix}, \qquad
\sigma^Z = \begin{pmatrix} 1 & 0 \\ 0 & -1 \end{pmatrix},  \]
and the auxiliary matrices
\[ \sigma^+ = \frac{1}{2} \left( \sigma^X +  \sqrt{-1} \sigma^Y \right) =  \begin{pmatrix} 0 & 1 \\ 0 & 0 \end{pmatrix}, \quad  
\sigma^- = \frac{1}{2} \left( \sigma^X -  \sqrt{-1} \sigma^Y \right) =  \begin{pmatrix}0 & 0 \\ 1 & 0 \end{pmatrix}. \]
From straightforward calculations the following statement is easily obtained (cf., e.g., \cite{Ne} for a similar statement).
%%%%%%%%%%%%%%%%%%%%%%%%%%%%%%%%%%%%%%%%%%%%%%%%%%%%%%%%%%%%%%%%%%%%%%%%
\begin{prop}
We have
\begin{equation}\label{XXZHamiltonian}
\begin{split}
H^\text{XXZ}_\text{bdy} &= \frac{1}{2} \Biggl( \sum_{i=1}^{n-1} \bigl( \sigma^X_i \sigma^X_{i\!+\!1} + \sigma^Y_i \sigma^Y_{i\!+\!1} +  \frac{\kappa+\kappa^{-1}}{2} \sigma^Z_i \sigma^Z_{i\!+\!1} \bigr) + \\
& \quad+ \frac{\kappa-\kappa^{-1}}{2} \biggl( \frac{(1+\kappa_0 \upsilon_0)(1-\kappa_0 \upsilon_0^{-1}) \sigma^Z_1 + 4\kappa_0\bigl(\psi_0 \sigma^+_1 + \psi_0^{-1} \sigma^-_1 \bigr)}{(1+\kappa_0 \upsilon_0^{-1})(1-\kappa_0 \upsilon_0)} + \\
& \qquad \qquad \quad - \frac{(1+\kappa_n \upsilon_n)(1-\kappa_n \upsilon_n^{-1}) \sigma^Z_n  - 4 \kappa_n \bigl( \psi_n^{-1} \sigma^+_n + \psi_n \sigma^-_n \bigr) }{(1+\kappa_n \upsilon_n^{-1})(1-\kappa_n \upsilon_n)} \biggr) \Biggr) + \hspace{-3mm} \\
 & \qquad +\tilde C \, \Id_{(\C^2)^{\otimes n}},
\end{split}
\end{equation}
where
\[ \tilde C =  \frac{\kappa-\kappa^{-1}}{2} \sum_{i \in \{0,n\}} \frac{1+\kappa_i^2}{(1-\kappa_i \upsilon_i)(1+\kappa_i \upsilon_i^{-1})} - \frac{n-1}{4} (\kappa+\kappa^{-1}). \]
\end{prop}
%%%%%%%%%%%%%%%%%%%%%%%%%%%%%%%%%%%%%%%%%%%%%%%%%%%%%%%%%%%%%%%%%%%%%%%%%%%%
\begin{rema}
The first line of \eqref{XXZHamiltonian} is the Hamiltonian of the Heisenberg XXZ spin-$\frac{1}{2}$ chain. The second and third lines describe explicit three-parameter
integrable boundary conditions for the left and right boundary of the spin chain, respectively.
\end{rema}
%%%%%%%%%%%%%%%%%%%%%%%%%%%%%%%%%%%%%%%%%%%%%%%%%%%%%%%%%%%%%

%%%%%%%%%%%%%%%%%%%%%%%%%%%%%%%%%%%%%%%%%%%%%%%%%%%%%%%%%%%%%%
\begin{prop} \label{Hamiltonian2BTL}
Write $d_i=1$ for $1 \leq i < n$ and for $i=0,n$ write
\[ d_i = \frac{-\kappa_i(\kappa \kappa_i^{-1} + \kappa^{-1} \kappa_i)}{(1-\kappa_i \upsilon_i)(1+\kappa_i \upsilon_i^{-1})}. \]
Then
\[ H^\text{XXZ}_\text{bdy} = \sum_{i=0}^n d_i \hat{\rho}^{\underline{\kappa}}_{\psi_0,\psi_n}(e_i). \] 
\end{prop}
%%%%%%%%%%%%%%%%%%%%%%%%%%%%%%%%%%%%%%%%%%%%%%%%%%%%%%%%%%%%%
\begin{proof}
This follows immediately from the observations that
\begin{gather*} 
 \check{r}'_{i \, i\!+\!1}(1) = \frac{1}{\kappa-\kappa^{-1}} \hat \rho(e_i) ,
\qquad
\frac{k'_n(1)}{2} = d_n \hat \rho(e_n) ,
\displaybreak[2] 
\\
\frac{\Tr_0\bigl( \theta_0 \bar k_0(\kappa^2) \theta_0 \check r'_{01}(1)\bigr)}{\Tr\bigl(\theta \bar k(\kappa^2) \theta\bigr)} = d_0 \hat \rho(e_0) + \frac{C_0}{\kappa-\kappa^{-1}}  \Id_{(\C^2)^{\otimes n}}. \qedhere
\end{gather*} 
\end{proof}
%%%%%%%%%%%%%%%%%%%%%%%%%%%%%%%%%%%%%%%%%%%%%%%%%%%%%%%%%%%%%%
\begin{rema}
With the notations of Proposition \ref{Hamiltonian2BTL}, the linear operator
\[
H^{\text{loop}}_{\text{bdy}}:=\sum_{i=0}^nd_i\omega^{\underline{\delta}}_{\beta_0,\beta_1}(e_i)
\]
on $\mathbb{C}[\mathcal{M}]$, with $\omega^{\underline{\delta}}_{\beta_0,\beta_1}$ the matchmaker representation and the parameters $\underline{\delta},\beta_0,\beta_1$
related to $\underline{\kappa},\psi_0,\psi_n$ by  \eqref{eq:deltas} and \eqref{parameterproduct}, is the quantum Hamiltonian of the Temperley-Lieb loop model (also known as the
dense loop model) with general three-parameter open boundary conditions on both the left and the right boundaries, cf., e.g., \cite{dG,dGN} for special cases. Theorem \ref{isomorphicrepresentations} and
Proposition \ref{Hamiltonian2BTL} establish the link between the XXZ spin-$\frac{1}{2}$ chain and the 
Temperley-Lieb loop model for general integrable boundary conditions, not only on the level of the quantum Hamiltonian but also on the level of the underlying quantum symmetry
algebra (the two-boundary Temperley-Lieb algebra). See \cite{MNGB} for a detailed discussion on Temperley-Lieb loop models and
their relation to XXZ spin chains.
\end{rema}
%%%%%%%%%%%%%%%%%%%%%%%%%%%%%%%%%%%%%%%%%%%%%%%%%%%%%%%%%%%%

%%%%%%%%%%%%%%%%%%%%%%%%%%%%%%%%%%%%%%%%%%%%
\section{Solutions of reflection quantum KZ equations}\label{Solsection}
%%%%%%%%%%%%%%%%%%%%%%%%%%%%%%%%%%%%%%%%%%%

%%%%%%%%%%%%%%%%%%%%%%%%%%%%%%%%%%%%%%%%%%%%%%%%%%
\subsection{The nonsymmetric difference Cherednik-Matsuo correspondence}\label{dCMsection}
%%%%%%%%%%%%%%%%%%%%%%%%%%%%%%%%%%%%%%%%%%%%%%%%%%

In this subsection we extend some of the main results of \cite{St1} to the Koornwinder setup. We do not dive into the detailed proofs,
which are rather straightforward adjustments of the proofs in \cite{St1}. We do give precise references to the corresponding statements in \cite{St1}. 

Theorem \ref{thmN} gives rise to an algebra homomorphism
\[
\varpi_{\underline{\kappa};\upsilon_0,\upsilon_n}: H(\underline{\kappa})\rightarrow\textup{End}_{\mathbb{C}}\bigl(\mathcal{M}(T)\bigr)
\]
defined by 
\[
\bigl(\varpi_{\underline{\kappa};\upsilon_0,\upsilon_n}(T_j)f\bigr)(\bm{t}):=\kappa_jf(\bm{t})+c_j(\bm{t};\underline{\kappa};\upsilon_0,\upsilon_n)(f(s_j\bm{t})-f(\bm{t})).
\]
Note that the subspace $\mathbb{C}[T]$ of regular functions on $T$ is an invariant subspace. The resulting $H(\underline{\kappa})$-module
is the analog of Cherednik's basic representation. For our purposes it is convenient to consider the basic representation for inverted
parameters $\underline{\kappa}^{-1},\upsilon_0^{-1},\upsilon_n^{-1}$ (the deformation parameter $q$ will always remain unchanged).
We therefore write $\varpi:=\varpi_{\underline{\kappa}^{-1};\upsilon_0^{-1},\upsilon_n^{-1}}$ for the associated representation map in the remainder of this section.

The operators $\varpi(Y^\lambda)$ ($\lambda\in\mathbb{Z}^n$) on $\mathcal{M}(T)$ are pairwise commuting $q$-difference reflection operators,
known as nonsymmetric Koornwinder operators. Their common eigenspaces are denoted by
\[
\textup{Sp}_K(\gamma;\underline{\kappa};\upsilon_0,\upsilon_n):=\{ f\in\mathcal{M}(T) \,\, | \,\, \varpi_{\underline{\kappa}^{-1};
\upsilon_0^{-1},\upsilon_n^{-1}}(Y_i)f=\gamma_if\quad \forall\, i=1,\ldots,n\}
\]
for $\gamma\in T$.
If $\gamma\in T_I^{\underline{\kappa}^{-1}}$ (see Definition \ref{basicdefprincipal}) then we have the natural subspace
\begin{equation*}
\begin{split}
\textup{Sp}_K^{I}(\gamma;\underline{\kappa};\upsilon_0,\upsilon_n):=&\{ f\in\mathcal{M}(T) \,\, | \,\, \varpi_{\underline{\kappa}^{-1};\upsilon_0^{-1},\upsilon_n^{-1}}(h)f=
\chi_{I,\gamma}^{\underline{\kappa}^{-1}}(h)f\qquad \forall\, h\in H_I(\underline{\kappa}^{-1})\}\\
=&\{ f\in \textup{Sp}_K(\gamma;\underline{\kappa};\upsilon_0,\upsilon_n) \,\, | \,\, \varpi_{\underline{\kappa}^{-1};\upsilon_0^{-1},\upsilon_n^{-1}}(T_i)f=\kappa_i^{-1}f\quad \forall i\in I\}
\end{split}
\end{equation*}
of the common eigenspace $\textup{Sp}_K^{\underline{\kappa}}(\gamma)$.
In writing these common eigenspaces we suppress the dependence on the parameters if no confusion can arise.
%%%%%%%%%%%%%%%%%%%%%%%%%%%%%%%%%%%%%%%%%%%%%%%%%%%%%%%%%%%%%%%%%%%%%%%%%%%%%%
\begin{eg}\label{nsbhf}
For $|q|<1$, for generic $\underline{\kappa},\upsilon_0,\upsilon_n$ and for $\gamma\in T$ satisfying
\begin{equation}\label{conditionsspecial}
\upsilon_0\upsilon_n^{-1}\gamma_i^{\pm 1}\not\in -q^{\frac{1}{2}+\mathbb{Z}_{\geq 0}}\qquad \forall\, i\in\{1,\ldots,n\},
\end{equation}
we write $\mathcal{E}_\gamma$ to be the nonsymmetric basic hypergeometric function associated to the Koornwinder root datum with multiplicity function 
$(k_0,k_{\vartheta},k_\theta,k_{2a_0},k_{2\theta})$ being $(\underline{\kappa}^{-1},\upsilon_0^{-1},\upsilon_n^{-1})$
(see \cite[Def. 2.14 \& \S 5.2]{St2} for the definition and the notations). Then $\mathcal{E}_\gamma\in\textup{Sp}_K(\gamma;\underline{\kappa};\upsilon_0,\upsilon_n)$.
Here the conditions \eqref{conditionsspecial} ensure that the nonsymmetric basic hypergeometric function may be specialized to $\gamma$ in its spectral parameter (cf. \cite[Thm. 2.13(ii)]{St2}).  

The transformation property \cite[Thm. 2.13(ii)(2)]{St2} of the nonsymmetric basic hypergeometric function with respect to the action of the double affine Hecke algebra implies that 
\[
\varpi_{\underline{\kappa}^{-1};\upsilon_0,\upsilon_n}(T_i)\mathcal{E}_\gamma=\kappa_i^{-1}\mathcal{E}_\gamma,\qquad i\in I
\] 
if $\gamma\in T_I^{\underline{\kappa}^{-1}}$ since the numerators of $c_i(\cdot; \upsilon_n,\kappa,\kappa_n;\upsilon_0,\kappa_0)$ vanish
if $\gamma\in T_I^{\underline{\kappa}^{-1}}$ and $i\in I$. Hence
\[
\mathcal{E}_\gamma\in\textup{Sp}_K^I(\gamma;\underline{\kappa};\upsilon_0,\upsilon_n)
 \]
if $\gamma\in T_I^{\underline{\kappa}^{-1}}$.
\end{eg}
%%%%%%%%%%%%%%%%%%%%%%%%%%%%%%%%%%%%%%%%%%%%%%%%%%%%%%%%%%%%%%%%%%%%%%%%%%%%%%%
The nonsymmetric difference Cherednik-Matsuo correspondence in the present setup gives a bijective correspondence between 
$\textup{Sp}_K^{I}(\gamma;\underline{\kappa};\upsilon_0,\upsilon_n)$ and the $W_0$-in\-va\-riant solutions of the reflection quantum KZ equations associated to 
the principal series module $M_I^{\underline{\kappa}}(\gamma)$.
The nonsymmetric difference Cherednik-Matsuo correspondence
was considered before in \cite{KT,St1,Ka} in different setups. We follow closely \cite{St1}.
To formulate the nonsymmetric difference Cherednik-Matsuo correspondence, we need to introduce a bit more notations and some basic facts on Coxeter groups and
Hecke algebras first. 

The affine Weyl group $W$, which is a Coxeter group with simple reflections $s_i$ ($i=0,\ldots,n$), has a length function $l: W\rightarrow \mathbb{Z}_{\geq 0}$ defined
as follows. The length of the unit element $e\in W$ is zero. For $w\in W$, the length $l(w)$ is the minimal number of simple reflections needed to write $w$ as product of
simple reflections. An expression $w=s_{i_1}\cdots s_{i_r}$ with $r=l(w)$ is called a reduced expression. The length function of the Weyl group $W_0$, viewed as Coxeter
group with simple reflections $s_i$ ($1\leq i\leq n$), coincides with the length function $l$ restricted to $W_0$.

Let $I\subseteq\{1,\ldots,n\}$. There is a distinguished set $W_0^I$ of coset representatives of $W_0/W_{0,I}$,  called the minimal coset representatives of $W_0/W_{0,I}$.
It is defined as 
\[
W_0^I:=\{w\in W_0 \,\, | \,\, l(wv)=l(w)+l(v)\quad \forall\, v\in W_{0,I}\}.
\]
There exist unique $w_0\in W_0$ and $w_{0,I}\in W_{0,I}$ of maximal length. In fact, $w_0$ simply acts on $\mathbb{R}^n$ by multiplication by $-1$.
For $I\subseteq\{1,\ldots,n\}$ we write $w_0^I:=w_0w_{0,I}^{-1}$. For $i\in I$ let $i_I^*\in\{1,\ldots,n\}$ be the unique index such that $w_0^Is_i=s_{i_I^*}w_0^I$.
Set $I^*:=\{i_I^*\,\, | \,\, i\in I\}$.

The case particularly relevant for our applications corresponds to the subset 
$J=\{1,\ldots,n-1\}$, in which case $w_{0,J}$ is the symmetric group element characterized by $w_{0,J}(\epsilon_i)=\epsilon_{n+1-i}$ for $i=1,\ldots,n$.
Thus $i_J^*=n-i$ ($1\leq i<n$) and $J^*=J$.

If $w=s_{i_1}\cdots s_{i_r}$ is a reduced expression,
then 
\[
T_w:=T_{i_1}T_{i_2}\cdots T_{i_r}
\]
is a well-defined element of the affine Hecke algebra $H(\underline{\kappa})$. The principal series $M_I^{\underline{\kappa}}(\gamma)=H(\underline{\kappa})\otimes_{H_I(\underline{\kappa})}
\mathbb{C}_{I,\gamma}$
for $\gamma\in T_I^{\underline{\kappa}}$ has a natural basis $\{v_w^{I}(\gamma;\underline{\kappa})\}_{w\in W_0^I}$ 
given by $v_w^{I}(\gamma;\underline{\kappa}):=\pi_{I,\gamma}^{\underline{\kappa}}(T_w)(1\otimes_{H(\underline{\kappa})}1)$.

The following theorem gives the nonsymmetric difference Cherednik-Matsuo correspondence.
%%%%%%%%%%%%%%%%%%%%%%%%%%%%%%%%%%%%%%%%%%%%%%%%%%%%%%%%%%%%%%%%%%%%%%%%
\begin{thm}\label{nsCMtheorem}
Let $I\subseteq\{1,\ldots,n\}$ and $\gamma\in T_I^{\underline{\kappa}}$. Then $w_0^I\gamma^{-1}\in T_{I^*}^{\underline{\kappa}^{-1}}$ and
the linear map $\mathcal{M}(T)\rightarrow \mathcal{M}(T)\otimes M_I^{\underline{\kappa}}(\gamma)$,
defined by
\[
\phi\mapsto \sum_{w\in W_0^I}\varpi_{\underline{\kappa}^{-1};\upsilon_0^{-1},\upsilon_n^{-1}}(T_{w(w_0^I)^{-1}})\phi\otimes v_w^I(\gamma;\underline{\kappa})
\]
restricts to a linear isomorphism 
\[\alpha_{\upsilon_0,\upsilon_n}^{\underline{\kappa}}: \textup{Sp}_K^{I^*}(w_0^I\gamma^{-1};\underline{\kappa};\upsilon_0,\upsilon_n)\overset{\sim}{\longrightarrow}
\textup{Sol}_{KZ}(M_I^{\underline{\kappa}}(\gamma);\underline{\kappa};\upsilon_0,\upsilon_n)^{W_0}.
\]
\end{thm}
%%%%%%%%%%%%%%%%%%%%%%%%%%%%%%%%%%%%%%%%%%%%%%%%%%%%%%%%%%%%%%%%%%%%%%%%%%
\begin{proof}
The proof is similar to the proof of \cite[Prop. 4.7]{St1}. We will sketch here the main steps.
Write $\nabla=\nabla^{M_I^{\underline{\kappa}}(\gamma)}$.
By a direct computation, analogous to the proof of \cite[Cor. 4.4]{St1},
one shows that a meromorphic $M_I^{\underline{\kappa}}(\gamma)$-valued function 
\[
\psi=\sum_{w\in W_0^I}\phi_w\otimes v_w^I(\gamma;\underline{\kappa})
\]
on $T$ is $\nabla(W_0)$-invariant if and only if 
\[
\phi_w=\varpi_{\underline{\kappa}^{-1};\upsilon_0^{-1},\upsilon_n^{-1}}(T_{w(w_0^I)^{-1}})\phi\qquad 
\forall\,w\in W_0^I
\]
with $\phi\in\mathcal{M}(T)$ satisfying the invariance property
\begin{equation}\label{phivar}
\varpi_{\underline{\kappa}^{-1};\upsilon_0^{-1},\upsilon_n^{-1}}(T_i)\phi=\kappa_i^{-1}\phi\qquad
\forall\, i\in I^*.
\end{equation}
So we need to investigate what the necessary and sufficient additional conditions on 
$\phi\in\mathcal{M}(T)$ satisfying \eqref{phivar} are to ensure that the associated 
$\nabla(W_0)$-invariant $M_I^{\underline{\kappa}}(\gamma)$-valued function
\[
\psi=\sum_{w\in W_0^I}
\varpi_{\underline{\kappa}^{-1};\upsilon_0^{-1},\upsilon_n^{-1}}(T_{w(w_0^I)^{-1}})\phi\otimes
v_w^I(\gamma;\underline{\kappa})
\]
on $T$ becomes $\nabla(W)$-invariant. The additional equation that $\psi$ should satisfy 
is $\nabla(s_0)\psi=\psi$. An algebraic computation analogous to the proof
of \cite[Prop. 4.7]{St1} shows that this is equivalent to the additional requirement
on $\phi$ that $\phi\in\textup{Sp}_K(w_0^I\gamma^{-1};\underline{\kappa};\upsilon_0,\upsilon_n)$.
This completes the proof of the theorem.
\end{proof}
%%%%%%%%%%%%%%%%%%%%%%%%%%%%%%%%%%%%%%%%%%%%%%%%%%%%%%%%%%%%%%%%%%%%%%%%%%

%%%%%%%%%%%%%%%%%%%%%%%%%%%%%%%%%%%%%%%%%%%%%%%%%%%%%%%%%%%%%%%
\begin{rema}\label{differenceremark}
{\bf (i)}
The difference Cherednik-Matsuo correspondence is a correspondence between the common eigenspace of the (higher order) Koornwinder $q$-dif\-fe\-ren\-ce operators \cite{Ko,N} and the full space $\textup{Sol}_{KZ}(M_\emptyset^{\underline{\kappa}}(\gamma);\underline{\kappa};\upsilon_0,\upsilon_n)$ of solutions of the reflection quantum KZ equations. 
It can be obtained from a spinor version of the nonsymmetric difference Cherednik-Matsuo correspondence by a symmetrization procedure, cf. \cite[\S 5]{St1}. 
A distinguished $W_0$-invariant solution of the spectral problem of the Koornwinder $q$-difference operators is the symmetric basic hypergeometric function $\mathcal{E}_\gamma^+$ of Koornwinder type, obtainable from $\mathcal{E}_\gamma$ (see Example \ref{nsbhf}) by a symmetrization procedure \cite[\S 2.6]{St2}. 
For $n=1$ the symmetric basic hypergeometric function $\mathcal{E}_\gamma^+$ is the Askey-Wilson function \cite{IR,KS, Stlink}, which is a nonpolynomial eigenfunction of the Askey-Wilson \cite{AW} second-order $q$-difference operator, alternatively expressible as a  very-well-poised ${}_8\phi_7$ series.\\
{\bf (ii)} For $\gamma\in T_I^{\underline{\kappa}}$ the canonical map $M_\emptyset^{\underline{\kappa}}(\gamma)\rightarrow M_I^{\underline{\kappa}}(\gamma)$
 induces a linear map 
 \begin{equation}\label{map}
 \textup{Sol}_{KZ}(M_\emptyset^{\underline{\kappa}}(\gamma);\underline{\kappa};\upsilon_0,\upsilon_n)\rightarrow
\textup{Sol}_{KZ}(M_I^{\underline{\kappa}}(\gamma);\underline{\kappa},\upsilon_0,\upsilon_n),
\end{equation}
cf, \cite[(4.15)]{St1}. 
Combined with {\bf (i)}, it leads to the construction of solutions of the reflection quantum KZ equations from common eigenfunctions of the (higher order) Koornwinder $q$-difference operators. 
It is an interesting open problem to understand the behaviour of the asymptotic basis of solutions of the reflection quantum KZ equations and their connection coefficients (see \cite{St3}) under the map \eqref{map}. 
It is expected to lead to elliptic solutions of dynamical Yang-Baxter equations and reflection equations, cf. \cite[\S 1.4]{St3}.
\end{rema}
%%%%%%%%%%%%%%%%%%%%%%%%%%%%%%%%%%%%%%%%%%%%%%%%%%%%%%%%%

%%%%%%%%%%%%%%%%%%%%%%%%%%%%%%%%%%%%%%%%%%%%%%%%%%%%%%%%%%%%%%
\subsection{Nonsymmetric Koornwinder polynomials}\label{Kosection}
%%%%%%%%%%%%%%%%%%%%%%%%%%%%%%%%%%%%%%%%%%%%%%%%%%%%%%%%%%%%%%
Let $\lambda\in\mathbb{Z}^n$ and define $\gamma_\lambda=(\gamma_{\lambda,1},\ldots,\gamma_{\lambda,n})\in T$ by
\[
\gamma_{\lambda,i}=q^{\lambda_i}(\kappa_0\kappa_n)^{-\eta(\lambda_i)}\bigl(\prod_{j<i}\kappa^{\eta(\lambda_j-\lambda_i)}\bigr)
\bigl(\prod_{j>i}\kappa^{-\eta(\lambda_i-\lambda_j)}\bigr)\bigl(\prod_{j\not=i}\kappa^{-\eta(\lambda_i+\lambda_j)}\bigr)
\]
where $\eta(x)=1$ if $x>0$ and $\eta(x)=-1$ if $x\leq 0$. Note that if $\lambda_1\leq\lambda_2\leq\cdots\leq\lambda_n\leq 0$ then
\[
\gamma_\lambda=\bigl(\kappa_0\kappa_n\kappa^{2(n-1)}q^{\lambda_1},\ldots,\kappa_0\kappa_n\kappa^2q^{\lambda_{n-1}},\kappa_0\kappa_nq^{\lambda_n}).
\]
Another special case is $\bm{m}=(m,\ldots,m)\in\mathbb{Z}^n$ with $m\in\mathbb{Z}$,
\begin{equation}\label{gammaspecial}
\gamma_{\bm{m}}=
\begin{cases}
(\kappa_0^{-1}\kappa_n^{-1}q^m,\kappa_0^{-1}\kappa_n^{-1}\kappa^{-2}q^m,\ldots,\kappa_0^{-1}\kappa_n^{-1}\kappa^{2(1-n)}q^m)\quad &\hbox{if }\, m\in\mathbb{Z}_{>0},\\
(\kappa_0\kappa_n\kappa^{2(n-1)}q^m,\kappa_0\kappa_n\kappa^{2(n-3)}q^m,\ldots,\kappa_0\kappa_nq^m)\quad &\hbox{if }\, m\in\mathbb{Z}_{\leq 0}.
\end{cases}
\end{equation}

In the remainder of the paper we assume that the parameters $\kappa_0,\kappa,\kappa_n,q$ are sufficiently generic,
meaning that $\gamma_\lambda\not=\gamma_\mu$ if $\lambda\not=\mu$. Then
\begin{equation}\label{Spintersection}
\textup{Sp}_K^{\emptyset}(\gamma_\lambda^{-1};\underline{\kappa};\upsilon_0,\upsilon_n)
\cap\mathbb{C}[T]=\textup{span}_{\mathbb{C}}\{P_\lambda\}
\end{equation}
is one-dimensional for all $\lambda\in\mathbb{Z}^n$. We can choose $P_\lambda=P_\lambda(\cdot;\underline{\kappa};\upsilon_0,\upsilon_n)$ such that the coefficient of
$\bm{t}^\lambda=t_1^{\lambda_1}t_2^{\lambda_2}\cdots t_n^{\lambda_n}$ in the expansion of $P_\lambda(\bm{t})$ in monomials
$\bm{t}^\mu$ ($\mu\in\mathbb{Z}^n$) is one. 
%%%%%%%%%%%%%%%%%%%%%%%%%%%%%%%%%%%%%%%%%%%%%%%%%%%%
\begin{defi}
$P_\lambda$ is called the monic nonsymmetric Koornwinder polynomial of degree $\lambda\in\mathbb{Z}^n$.
\end{defi}
%%%%%%%%%%%%%%%%%%%%%%%%%%%%%%%%%%%%%%%%%%%%%%%%%%%%%%

%%%%%%%%%%%%%%%%%%%%%%%%%%%%%%%%%%%%%%%%%%%%%%%%%%%%
Since $\mathbb{C}[T]=\bigoplus_{\lambda\in\mathbb{Z}^n}\mathbb{C}P_\lambda$ the discrete set $\{\gamma_\lambda^{-1}\}_{\lambda\in\mathbb{Z}^n}\subset T$ is
called the polynomial spectrum of the $q$-difference reflection operators $\varpi_{\underline{\kappa}^{-1};\upsilon_0^{-1},\upsilon_n^{-1}}(Y_i)$ ($1\leq i\leq n$). 
%%%%%%%%%%%%%%%%%%%%%%%%%%%%%%%%%%%%%%%%%%%%%%%%%%%%%%%%%%%%%%%%%
\begin{rema}\label{polred}
By \cite[Thm. 6.9]{St0}, the proof of \cite[Thm. 2.13(ii)]{St2} and \cite[(2.5)]{St2}, the nonsymmetric Koornwinder polynomial of degree $\lambda$ equals $\mathcal{E}_{\gamma_\lambda^{-1}}$ (up to a multiplicative constant), with $\mathcal{E}_\gamma$ the nonsymmetric basic hypergeometric function as discussed in Example \ref{nsbhf}. 
\end{rema}
%%%%%%%%%%%%%%%%%%%%%%%%%%%%%%%%%%%%%%%%%%%%%%%%%%%%

The following lemma is convenient for later purposes.
%%%%%%%%%%%%%%%%%%%%%%%%%%%%%%%%%%%%%%%%%%%%%%%%%%%%
\begin{lem}\label{invariancelem}
Let $i\in\{1,\ldots,n\}$. The following two conditions are equivalent.
\begin{enumerate}
\item[{\bf (i)}] $s_i\lambda=\lambda$.
\item[{\bf (ii)}] $\varpi(T_i)P_\lambda=\kappa_i^{-1}P_\lambda$.
\end{enumerate}
If $I\subseteq\{1,\ldots,n\}$ and $s_i\lambda=\lambda$ for all $i\in I$, then
$P_\lambda\in \textup{Sp}_K^I(\gamma_\lambda^{-1};\underline{\kappa};\upsilon_0,\upsilon_n)$. 
\end{lem}
%%%%%%%%%%%%%%%%%%%%%%%%%%%%%%%%%%%%%%%%%%%%%%%%%%%%%%
\begin{proof}
The first part of the lemma is \cite[Prop. 4.15]{StBook}. 
For the second part, suppose that $I$ is a subset of $\{1,\ldots,n\}$ and suppose that $\lambda\in\mathbb{Z}^n$ satisfies $s_i\lambda=\lambda$ for all $i\in I$.
Then \cite[Prop. 3.5]{StBook} shows that $\gamma_\lambda^{-1}\in T_I^{\underline{\kappa}^{-1}}$. The statement then follows from the first part of the lemma.
\end{proof}
%%%%%%%%%%%%%%%%%%%%%%%%%%%%%%%%%%%%%%%%%%%%%%%%%%%%%%

For a more detailed discussion on nonsymmetric Macdonald-Koornwinder polynomials see, e.g., \cite{StBook}.

%%%%%%%%%%%%%%%%%%%%%%%%%%%%%%%%%%%%%%%%%%%%%%%%%%%%%%%
\begin{rema}
%differenceremark}
A suitable symmetrized version of the monic nonsymmetric Koornwinder polynomials give the monic Koornwinder polynomials \cite{Ko}. 
They are $W_0$-invariant Laurent polynomials in the variables $t_1,\ldots,t_n$, and common eigenfunctions of the (higher order) Koornwinder $q$-difference operators \cite{Ko,N}. 
For $n=1$ the Koornwinder $q$-difference operator is the Askey-Wilson second-order $q$-difference operator \cite{AW} and the Koornwinder polynomials are the celebrated monic Askey-Wilson \cite{AW} polynomials, see, e.g., \cite[\S 3.8]{StBook} for a detailed discussion.
\end{rema}
%%%%%%%%%%%%%%%%%%%%%%%%%%%%%%%%%%%%%%%%%%%%%%%%%%%%%%%%

%%%%%%%%%%%%%%%%%%%%%%%%%%%%%%%%%%%%%%%%%%%%%%%%%%%%%%%
\subsection{Laurent polynomial solutions of the reflection quantum KZ equations}
%%%%%%%%%%%%%%%%%%%%%%%%%%%%%%%%%%%%%%%%%%%%%%%%%%%%%%%%

Let $I\subseteq\{1,\ldots,n\}$ and $\gamma\in T_I^{\underline{\kappa}}$. Then $w_0^I\gamma^{-1}\in T_{I^*}^{\underline{\kappa}^{-1}}$, see Theorem \ref{nsCMtheorem}(i).
Hence for generic $\upsilon_0,\upsilon_n\in\mathbb{C}^*$ and $|q|<1$,
\[
\textup{Sp}_K^{I^*}(w_0^{I}\gamma^{-1}; \underline{\kappa};\upsilon_0,\upsilon_n)\not=\{0\}
\]
in view of Example \ref{nsbhf} (the same is true for $|q|>1$, using the nonsymmetric basic hypergeometric function for $|q|>1$ as constructed in \cite{St0}).
By Theorem \ref{nsCMtheorem} we conclude that nontrivial $W_0$-invariant solutions of the reflection quantum KZ equations associated to $(M_I^{\underline{\kappa}}(\gamma),\underline{\kappa},\upsilon_0,\upsilon_n)$ generically exist. 
In this subsection we focus on the $W_0$-invariant {\it Laurent polynomial} solutions of the reflection quantum KZ equations. 

\begin{defi}
Let $V$ be a $H(\underline{\kappa})$-module and $\upsilon_0,\upsilon_n\in\mathbb{C}^*$. 
We say that $(V,\upsilon_0,\upsilon_n)$ admits $W_0$-invariant Laurent polynomial solutions of the reflection quantum KZ equations if 
\[
\textup{Sol}_{KZ}(V;\underline{\kappa};\upsilon_0,\upsilon_n)^{W_0}\cap\bigl(\mathbb{C}[T]\otimes V\bigr)\not=\{0\}.
\]
A nonzero $f$ from this space is called a nontrivial $W_0$-invariant Laurent polynomial solution of the reflection quantum KZ equations associated to $(V,\upsilon_0,\upsilon_n)$.
\end{defi}

Recall that we assume that the multiplicity parameters are generic, i.e. that $\gamma_\lambda\not=\gamma_\mu$ if $\lambda\not=\mu$.
%%%%%%%%%%%%%%%%%%%%%%%%%%%%%%%%%%%%%%%%%%%%%%%%%%%%%%%
\begin{prop}\label{polprop}
Let $I\subseteq\{1,\ldots,n\}$ and $\gamma\in T_I^{\underline{\kappa}}$. Then $(M_I^{\underline{\kappa}}(\gamma),
\upsilon_0,\upsilon_n)$ admits $W_0$-invariant Laurent polynomial solutions of the reflection quantum KZ equations
if and only if $w_0^I\gamma=\gamma_\lambda$
with $\lambda\in\mathbb{Z}^n$ satisfying $s_i\lambda=\lambda$ for all $i\in I^*$. The associated nontrivial $W_0$-invariant Laurent
polynomial solutions of the reflection quantum KZ equations are the nonzero multiples of 
\[
\alpha_{\upsilon_0,\upsilon_n}^{\underline{\kappa}}(P_\lambda)=\sum_{w\in W_0^I}\varpi_{\underline{\kappa}^{-1};\upsilon_0^{-1},\upsilon_n^{-1}}(T_{w(w_0^I)^{-1}})P_\lambda\otimes v_w^I(\gamma;\underline{\kappa}).
\]
\end{prop}
%%%%%%%%%%%%%%%%%%%%%%%%%%%%%%%%%%%%%%%%%%%%%%%%%%%%%%%%
\begin{proof}
This follows from Theorem \ref{nsCMtheorem}, Lemma \ref{invariancelem},  \eqref{Spintersection} and the fact
that $\mathbb{C}[T]=\bigoplus_{\lambda\in\mathbb{Z}^n}\mathbb{C}P_\lambda$.
\end{proof}
%%%%%%%%%%%%%%%%%%%%%%%%%%%%%%%%%%%%%%%%%%%%%%%%%%%%%%%%

We now come to our main application of Proposition \ref{polprop}, by applying it to the spin representations $\rho_{\psi_0,\psi_n}^{\underline{\kappa}}\simeq \pi_{J,\zeta}^{\underline{\kappa}}$. 
Recall here that $J=\{1,\ldots,n-1\}$ and that
\[
\zeta=(\psi_0\psi_n\kappa^{n-1},\psi_0\psi_n\kappa^{n-3},\ldots,\psi_0\psi_n\kappa^{1-n}),
\]
see Proposition \ref{rhopiequivalence}.

Let $\lambda\in\mathbb{Z}^n$. Then $s_i\lambda=\lambda$ for all $i\in J^*$ if and only if $\lambda=\bm{m}$ for some $m\in\mathbb{Z}$.
On the other hand, note that
\[
w_0^J\zeta=(\psi_0^{-1}\psi_n^{-1}\kappa^{n-1},\psi_0^{-1}\psi_n^{-1}\kappa^{n-3},\ldots,\psi_0^{-1}\psi_n^{-1}\kappa^{1-n}).
\]
%%%%%%%%%%%%%%%%%%%%%%%%%%%%%%%%%%%%%%%%%%%%%%%%%%%%%%%%%%%%%%%%%%%%
\begin{thm} \label{mainthm}
For generic parameters, the spin representation $(M_J^{\underline{\kappa}}(\zeta),\upsilon_0,\upsilon_n)$ admits nontrivial $W_0$-invariant Laurent polynomial solutions if and only if 
\begin{equation}\label{mcondition}
\psi_0\psi_nq^m=\bigl(\kappa_0\kappa_n\kappa^{n-1}\bigr)^{\eta(m)}
\end{equation}
for some $m\in\mathbb{Z}$ where, recall, $\eta(x)=1$ if $x>0$ and $\eta(x)=-1$ if $x\leq 0$.
The associated nontrivial $W_0$-invariant Laurent polynomial solutions of the reflection quantum KZ equations are multiples of
\[
\alpha_{\upsilon_0,\upsilon_n}^{\underline{\kappa}}(P_{\bm{m}})=
\sum_{w\in W_0^J}\varpi_{\underline{\kappa}^{-1};\upsilon_0^{-1},\upsilon_n^{-1}}(T_{w(w_0^J)^{-1}})P_{\bm{m}}\otimes v_w^I(\zeta;\underline{\kappa}).
\]
\end{thm}
%%%%%%%%%%%%%%%%%%%%%%%%%%%%%%%%%%%%%%%%%%%%%%%%%%%%%%%%%%%%%%%%%%%%
\begin{proof}
In view of the previous proposition, it suffices to note that for a given $m\in\mathbb{Z}$ we have $w_0^J\zeta=\gamma_{\bm{m}}$ if and only if \eqref{mcondition} holds true.
\end{proof}

Different examples of $W_0$-invariant Laurent polynomial solutions of reflection quantum KZ equations are given in \cite{Ka,ZJ1,dFZJ,dGPS,PoThesis}.

We finally stress the importance of polynomial solutions of the quantum KZ equations, cf. \cite{dFZJ,dFZJ2,ZJ2,Pa}, in particular in relation to the Razumov-Stroganov conjectures \cite{RS} (recently proved by direct combinatorial methods in \cite{CS}). Similarly, polynomial solutions of the reflection quantum KZ equations play a similar role for the refinements of the Razumov-Stroganov conjectures for open boundaries from \cite{dGR}, cf., e.g., \cite{dGP2, ZJ1,KaPrep}.  We remark though that in the context of the Razumov-Stroganov conjectures the double affine Hecke algebra parameters $q,\underline{\kappa},\upsilon_0,\upsilon_n$ are specialized to particular non-generic values, in contrast to the setup of Theorem \ref{mainthm}, where the parameters $q,\underline{\kappa},\upsilon_0,\upsilon_n$ are assumed to be generic.

%%%%%%%%%%%%%%%%%%%%
%%                                                   %%
%%                References                 %%
%%                                                   %%
%%%%%%%%%%%%%%%%%%%%

\end{document}